\documentclass[12pt]{amsart}

\usepackage{amsmath,amsthm,amscd,euscript,amssymb}
\def \r{\mathbb R}

\def \q{\mathbb Q}

\def \L{\mathcal L}
\def \T{\mathcal T}
\def \hf{\hbox{\rm H}\Phi}
\def \gr{\hbox{\rm Gr}(1,\r P^2)}

\newtheorem{theorem}{Theorem}[section]
\newtheorem{lemma}[theorem]{Lemma}

\newtheorem{proposition}[theorem]{Proposition}
\newtheorem{corollary}[theorem]{Corollary}
\theoremstyle{remark}
\newtheorem{remark}[theorem]{Remark}
\theoremstyle{definition}
\newtheorem{definition}[theorem]{Definition}
\newtheorem{example}[theorem]{Example}
\newtheorem{problem}{Problem}
\newtheorem{conjecture}[problem]{Conjecture}

\title{The combinatorial geometry of stresses in frameworks}

\author{Oleg Karpenkov}

\date{8 December 2015}


\email[Oleg Karpenkov]{karpenk@liv.ac.uk}

\begin{document}
\input{epsf}

\begin{abstract}
In this paper we formulate and prove necessary and sufficient geometric conditions
for existence of generic tensegrities in the plane for arbitrary graphs.
The conditions are written in terms of ``meet-join'' relations for the configuration spaces
of fixed points and non-fixed lines through fixed points.
\end{abstract}

\maketitle
\tableofcontents

\section*{Introduction}

Given an arbitrary graph $G$. {\it How to decide which of the realizations of $G$
in the plane admits a non-zero equilibrium self-stress and which does not?}
The main goal of this paper is to construct a necessary and sufficient realizability condition in 
terms of geometric operations on  point-line configurations related to graph realizations.

\vspace{2mm}

{\noindent
{\bf Background.}
J.C.~Maxwell was one of the first to consider equilibrium states for frames
under the action of static forces, see~\cite{Max}.
Later in the second half of the twentieth century the artist K.~Snelson built many
surprising cable-bar sculptures that are actually such frames in equilibrium, see~\cite{Sne}.
Further R.~Buckminster Fuller introduced the term ``tensegrity'' for
these constructions, as a combination of words ``tension'' and
``integrity''. In his book~\cite{Mot} R.~Motro gives an overview
of the history of tensegrity constructions.
}

Tensegrities have a wide range of applications in different
branches of science. For instance they are
used in the study of cells~\cite{Ing,Ing2},
viruses~\cite{Cas,Sim},
for construction of deployable mechanisms~\cite{Ske,Tib}, etc.
Tensegrities often show up in architecture (as a light structural support) and in arts,
e.g., in 2015 a new tensegrity sculpture {\it TensegriTree} was erected (see~\cite{Gra})
to celebrate the 50th Anniversary of the University of Kent.

In mathematics, tensegrities were investigated in many papers.
B.~Roth and W.~Whiteley in~\cite{Rot1}
R.~Connelly and W.~Whiteley in~\cite{Con}
and
R.~Connelly in~\cite{Con2} studied rigidity and flexibility of
tensegrities (see also the survey about rigidity in~\cite{Whi}).
Minimal rigidity questions where studied by
M.~Wang, M.~Sitharam in~\cite{Wan}.
In~\cite{JJSS2015} B.~Jackson, T.~Jord\'an, B.~Servatius and H.~Servatius discussed mechanical properties of tensegrities.
See also a very nice short introductory paper on tensegrities in mathematics by R.~Connelly~\cite{Con2013}.

Generalization of tensegrities were considered by
F.V.~Saliola and W.~Whiteley~\cite{SW} (spheric and projective geometries)
by D.~Kitson, S.C.~Power in~\cite{Kit2014}
(non-Euclidean geometry), by D.~Kitson and B.~Schulze in~\cite{Kit2015} (normed spaces),
by B.~Jackson and A.~Nixon in~\cite{JN2015} (on surfaces in~$\r^3$), etc.

\vspace{2mm}

{\noindent
{\bf Realizability conditions of tensegrities.}
In~\cite{WW,WW2} N.L.~White and W.~Whiteley introduced algebraic conditions for the existence of nonzero tensegrities.
They have expressed these conditions in terms of bracket ring using the determinants of extended rigidity matrices
(see also~\cite{White}).
In particular, they have examined many examples of graphs and write the corresponding conditions both in terms of bracket ring and
in terms of Cayley algebra.
In the papers~\cite{Guz2,Guz3,Guz1} M.~de~Guzm\'an and D.~Orden introduce atom decomposition techniques and
describe several conditions for some further examples of graphs.
}

\vspace{1mm}

All the studied examples suggest a simple geometric description of all tensegrities   
in terms of geometric ``meet-join'' operations of Cayley algebra.
It was not known how to construct geometric conditions from the corresponding bracket ring expressions
(in particular this involves the general open problem on bracket ring expression factorization). 

\vspace{1mm}

In paper~\cite{DKS} we have made first steps in the study of geometric conditions.
We have found two main surgeries and used them in order to find all the conditions for codimension one strata for graphs with
8 or smaller vertices.
In paper~\cite{KSS} we have studied the topological properties of the configuration spaces of all tensegrities for graphs
with 4 and 5 vertices.

\vspace{1mm}

Our main aim for the current paper is to develop a systematics geometric approach for arbitrary graphs in the plane.
In this paper we construct geometric conditions for tensegrities in the plane for arbitrary graphs
(see Theorem~\ref{tensegrity=condi} and Subsection~\ref{techniques}).
The conditions are written in terms of geometric operations (slightly extended ``meet-join'' operations of Cayley algebra) 
on lines and points of
configuration spaces defined by graph realizations in the plane.

\vspace{2mm}

{\noindent
{\bf Organization of the paper.}
We start in Section~1 with necessary definitions and preliminary discussions.
Further in Section~2 we introduced a projective extension of the classical theory of tensegrities.
Notice that the mentioned projective version precisely coincides with the classical theory once restricted to some affine chart.
In this section we study basics of projective statics and give necessary definitions regarding projective tensegrities.
In Subsection~2.5 we formulate the main result of this paper (see Theorem~\ref{tensegrity=condi}).
Further in Section~3 we consider framed cycles and define monodromy operators for them.
We prove that the monodromy operators for a framed cycle in general position are trivial if and only if
there exists a nonzero equilibrium force load on it (see Theorem~\ref{force=monodromy}).
In Section~4 we describe the techniques of resolution schemes in order to study
local properties of equilibrium force-loads at vertices.
Further in Section~5 we introduce the notion of quantizations and prove a necessary and sufficient
condition of existence of non-parallelizable tensegrities in terms of quantizations (see Theorem~\ref{tensegrity=quantiza}).
In Section~6 we prove the main result of the current paper on the necessary
and sufficient geometric conditions of existence of non-parallelizable tensegrities (i.e., Theorem~\ref{tensegrity=condi}).
In this section we also introduce a techniques to write these conditions explicitly.
We conclude this paper in Section~7 with a conjecture on strong geometric conditions
for nonzero tensegrities and some related discussions.
}

\vspace{2mm}

{\noindent
{\bf Acknowledgement.}
The author is grateful to Walter Whiteley and Meera Sitharam for helpful remarks and discussions.
}

\section{Setting up the problem}

We start with a the following definition of a framework in the Euclidean plane.
Later in Subsection~\ref{ProjTens} we switch to projective version of this definition.

\begin{definition}\label{defR}
Let $G$ be an arbitrary graph without loops and multiple edges on $n$ vertices.
Let $V(G)$ and $E(G)$ denote the sets of vertices and edges for $G$ respectively.
\itemize

\item{A {\it framework} $G(P)$ in the plane is a map of the graph $G$ with vertices $v_1, \ldots, v_n$ on
a finite point configuration $P=(p_1,\ldots,p_n)$ in $\r^2$ with
straight edges, such that $G(P)(v_i)=p_i$ for $i=1,\ldots,n$.}

\item{A {\it stress} $w$ on a framework is an assignment of real scalars
$w_{i,j}$ (called {\it tensions}) to its edges $p_ip_j$. We also
put $w_{i,j}=0$ if there is no edge between the corresponding
vertices. Note that in this notation we have $w_{i,j}=w_{j,i}$.}

\item{A stress $w$ is called a {\it self-stress} if the
following equilibrium condition is fulfilled at every vertex
$p_i$:
$$
\sum\limits_{\{j|j\ne i\}} w_{i,j}(p_i-p_j)=0.
$$
By $p_i-p_j$ we denote the vector from the point $p_i$ to the
point $p_j$.}
\item{A pair $(G(P),w)$ is called a {\it tensegrity} if
$w$ is a self-stress  for the framework $G(P)$.}
\end{definition}

\begin{example}\label{SectionDesargues}
Let us work through a particular example of tensegrities on 6 points with the graph $G$ shown on Figure~\ref{pic2.1} on the left.
What are all 6-point configurations $P$ in the plane whose frameworks $G(P)$ have nonzero self-stresses?
The complete answer to this question can be written as follows:

{\bf Generic cases:}

1) the lines $p_1p_2$, $p_3p_4$, $p_5p_6$ have a common point (see Figure~\ref{pic2.1} in the middle);

2) the lines $p_1p_2$, $p_3p_4$, $p_5p_6$ are parallel (see Figure~\ref{pic2.1} on the right);

{\bf Non-generic cases:}

3) the points $p_1$, $p_4$, $p_5$ are in a line;

4) the points $p_2$, $p_3$, $p_6$ are in a line;

5) the points $p_1$, $p_2$, $p_3$, $p_4$ are in a line;

6) the points $p_1$, $p_2$, $p_5$, $p_6$ are in a line;

7) the points $p_3$, $p_4$, $p_5$, $p_6$ are in a line.

\begin{figure}
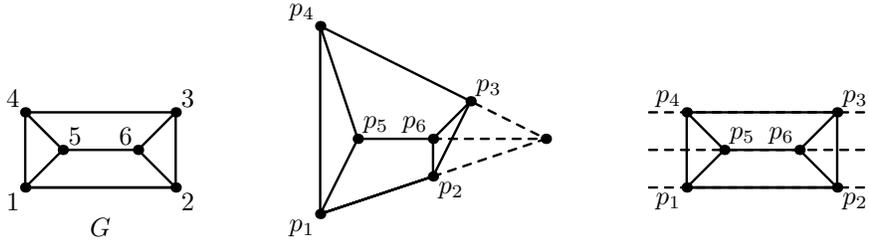

$$\epsfbox{pic2.1} \qquad\quad \epsfbox{pic2.3}\qquad\quad \epsfbox{pic2.2}$$
\caption{The graph $G$ (on the left) and two generic configurations
of existence of nonzero equilibrium stresses for $G$ (in the middle and on the right).}
\end{figure}\label{pic2.1}
\end{example}

\begin{remark}
Notice that the first geometric condition defines a Desargues configuration
$(p_1,\ldots,p_6, q, r_1, r_2, r_3)$,
where

--- $q$ is the intersection point of $p_1p_2$, $p_3p_4$, $p_5p_6$ ({\it center of perspective});

--- $r_1$ is the intersection point of the lines $p_1p_4$ and $p_2p_3$;

--- $r_2$ is the intersection point of the lines $p_4p_5$ and $p_2p_6$;

--- $r_3$ is the intersection point of the lines $p_1p_5$ and $p_3p_6$.
\\
By Desargues' theorem the points $r_1$, $r_2$, and $r_3$ are in a line, this line is called the {\it axis of perspectivity}.
\end{remark}

From this example it is clear that
the enumeration of all different cases is unobservable even for graphs with relatively small number of edges.
The non-generic cases are not as interesting, since non-generic frameworks admit tensegrities with zero
stresses at some edges.
It is better to study non-generic cases separately while examining subgraphs of the graph.
Here we face with the following nontrivial question:
{\it What are the non-generic configurations and how to define them geometrically?}

\begin{remark}
While constructing actual tensegrities in real life it is usually expected that the stresses are of the same level of magnitude.
Then one should not consider some neighborhood of the non-generic case.
\end{remark}

Suppose now we are in a generic situation.
We still have the following problem: {\it the number of generic cases has more that exponential
growth in terms of the number of edges in graphs}.

\vspace{1mm}

So we face with the following two difficulties:

{\bf I.} {\it How to remove non-generic cases?}

{\bf II.} {\it How to merge generic cases?}

\vspace{1mm}

In order to answer the first questions we provide the description of generic cases in Subsection~\ref{fimgp}.
The second question is resolved by projectivization of tensegrities.
In order to implement this approach we introduce a projective version of classical statics,
where forces in the projective plane are represented by 2-forms in $\r^3$ (see more detailed in Subsection~\ref{Static}).

\begin{remark}
In the context of this discussion we would like to mention a preprint~\cite{SW}
regarding spherical and hyperbolic tensegrities and their interrelation to Euclidean case.
Spherical tensegrities is a good alternative candidate to avoid the problem of parallelity.
\end{remark}

\begin{remark}\label{higher}
It is not much known about geometric description of nonzero tensegrities in the multidimensional case.
There are algebraic conditions in terms of bracket rings are provided in~\cite{Whi} by N.L.~White and W.~Whiteley,
but their factorization is known to be a hard problem.
In this context we also would like to mention the paper~\cite{Che}
by J.~Cheng, M.~Sitharam, I.~Streinu where
nonzero stress on an arbitrary framework
with a geometrically specific extension are studied.

In this paper we introduce a techniques of graph quantizations and a techniques of $\hf$-surgeries.
The first has a straightforward generalization to multidimensional case, while the second one is essentially planar.
One should invent some $n$-dimensional analogs of  $\hf$-surgeries in order to approach the combinatorial geometry of tensegrities in higher dimensions.
\end{remark}

\section{Projective geometry and tensegrities}

In this section we collect several notions and definitions that we use further to formulate the main result.
Followed by basic definitions of projective statics of Subsection~2.1
we provide with necessary definitions of projective tensegrities in Subsection~2.2.
Further in Subsection~2.3 we describe general geometric operations and relations on configurations of points and lines.
Next we define frameworks in general position in Subsection~2.4.
Finally, in Subsection~2.5 we formulate the main result of this paper.
We conclude Section~2 with geometric surgeries on frameworks that does not change the property to admit a nonzero tensegrity.

\subsection{Basics of projective statics}\label{Static}
Let us introduce an extension of classical statics to the projective case.
As we will see, the projective definition of forces coincides with the Euclidean definition of forces once restricted to an affine chart.

\subsubsection{Notions and definitions}

Denote by $\Lambda^2(\r^{k+1})$ the space of exterior 2-forms on $\r^{k+1}$.

\begin{definition}
A {\it force} acting on a rigid body in $\r P^k$ is a decomposable 2-form in $\Lambda^{2}(\r^{k+1})$.
\end{definition}

The sum of forces is defined as a sum of $2$-forms in the linear space $\Lambda^{2}(\r^{k+1})$.
If $k=2$ then the sum of any number of forces is a force, since all 2-forms in $\Lambda^{2}(\r^{3})$ are decomposable.
In fact, $k=2$ is the most important case for us
since we work in the projective 2-plane. If $k>2$, then the sum of two forces is not always a force, one can say that
it is a {\it state}, i.e., a 2-form that might be non-decomposable.

\vspace{1mm}

For a point $p=(a_1,\ldots,a_{k+1})$  set
$$
dp=a_1dx_1+\ldots + a_{k+1}dx_{k+1}.
$$
We also write
$$
\tilde p=(a_1:\ldots:a_{k+1}).
$$

\begin{definition}
Let $F$ be a nonzero force.
By definition, $F$ is a decomposable 2-form.
So there exists a pair of nonzero vectors $p_1,p_2\in \r^{k+1}$ such that $F=dp_1\wedge dp_2$.
\begin{itemize}
\item
Let $\tilde p_1,\tilde p_2$ be the projective points corresponding to $p_1$ and $p_2$.
The projective line $\tilde p_1\tilde p_2$ is the {\it line of force}.

\item
Consider an affine chart given by the line at infinity: $A_1x_1+\ldots+A_{k+1}x_{k+1}=0$.
Then the {\it vector of force} in this chart is the 1-form defined as 
the interior product of $F$ with the constant vector field $V$ defined as 
$$
\frac{\partial}{\partial t} (x_1,\ldots,x_{k+1})=(A_1,\ldots,A_{k+1}).
$$
Namely the vector of force is the following 1-form 
$$
\iota_V F ,\quad \hbox{where $\iota_V F(*)=F(V,*)$}.
$$
\end{itemize}
\end{definition}

\subsubsection{Lines of force and the parallelogram rule for force summation}

In order to establish the link to classical mechanics we show how to add forces.
First of all we write down the line of the force for $F_1+F_2$.
Further we examine the parallelogram rule for vector of forces summation. 

\begin{proposition}
Consider two nonzero forces
$$
F_1=dp_1\wedge dq_1 \quad \hbox{and} \quad  F_2=dp_2\wedge dq_2.
$$
Let the intersection of the lines of forces for $F_1$ and $F_2$ is a unique point  $r$.
Then the line of forces for $F_1+F_2$ passes through $r$. 
\end{proposition}

\begin{proof}
Without loss of generality we assume that $r\ne q_1$ and $r\ne q_2$.
Then there exist real numbers $\alpha$ and $\beta$ such that
$$
F_1=\alpha_1 dr\wedge dq_1 \quad \hbox{and} \quad  F_2=\alpha_2 dr\wedge d q_2.
$$
Therefore,
$$
F_1+F_2=dr\wedge d(\alpha_1q_1+\alpha_2q_2).
$$
Since $F_1$ and $F_2$ are nonzero forces with distinct lines of forces,
the point $\alpha_1q_1+\alpha_2q_2$ is nonzero and is not proportional to $r$.
Hence the line of forces for $F_1+F_2$ passes through $r$. 
\end{proof}

\begin{remark}
Notice that the {\it parallelogram rule} for force summation follows directly from the fact that
$$
\iota_V (F_1+F_2)=\iota_V (F_1)+\iota_V (F_2).
$$
\end{remark}

\subsubsection{Links to classical mechanics}
Let us show how the classical statics is embedded into projective statics.
First of all we consider the affine chart $x_{n+1}=0$, all points in this chart
are in one-to-one correspondence with points of type
$$
(u_1,\ldots, u_k,1).
$$
Any force $F$ defined by two points of this affine chart can be written as
$$
F=\lambda d\big((a_1,\ldots,a_k,1)\big)\wedge d\big( (b_1,\ldots,b_k,1)\big).
$$
The vector of force $F$ is, therefore, as follows 
$$
\iota_V F=\lambda\Big(d\big( (b_1,\ldots,b_k,1)\big)-d\big( (a_1,\ldots,a_k,1) \big)\Big)=
\lambda d\big( (b_1-a_1,\ldots,b_k-a_k,0)\big).
$$
In classical statics one considers the first $k$ coordinates of this form as the 
corresponding {\it vector of force}:
$$
\lambda(b_1-a_1,\ldots,b_k-a_k).
$$
In case if $\lambda\ne 0$, the {\it line of force} in classical statics is the line passing through
$(a_1,\ldots,a_k)$ and $(b_1,\ldots,b_k)$.

\subsection{Projective definition of tensegrity}\label{ProjTens}

In order to prevent considering different cases of parallel/non-parallel lines,
we continue to work in the projective settings.
In this paper we use the following projective definition of a tensegrity.

\begin{definition}
Let $G$ be an arbitrary graph without loops and multiple edges.
Let $n$ be the number of vertices of $G$.

\vspace{1mm}

\itemize
\item{A {\it framework} $G(P)$ in $\r P ^2$ is a map of the graph $G$ with vertices $v_1, \ldots, v_n$ on
a finite point configuration $P=(p_1,\ldots,p_n)$ in $\r P^2$ with
straight edges, such that $G(P)(v_i)=p_i$ for $i=1,\ldots,n$.}

\vspace{1mm}

\item{A {\it stress} on an edge $p_ip_j$ is an assignment of two forces $F_{i,j}$ and $F_{j,i}$
whose line of forces coincide with the line $p_ip_j$ such that $F_{i,j}+F_{j,i}=0$.
}

\vspace{1mm}

\item{A {\it force-load} $F$ on a framework is an assignment of stresses  for all edges.
Additionally we set $F_{i,j}=0$ if $v_iv_j$ is not an edge of $G$.
}

\vspace{1mm}

\item{A force-load $F$ is called an {\it equilibrium} force-load if, in addition, the
following equilibrium condition is fulfilled at every vertex
$p_i$:
$$
\sum\limits_{\{j|j\ne i\}} F_{i,j}=0.
$$
}

\vspace{1mm}

\item{A pair $(G(P),F)$ is called a {\it tensegrity} if
$F$ is an equilibrium force-load for the framework $G(P)$.}
\end{definition}

\begin{remark}
Once restricted to an affine chart we arrive to Definition~\ref{defR}.  
Here for all admissible pairs $(i,j)$ we have 
$$
\iota_VF_{i,j}=\omega_{i,j} (dp_i-dp_j),
$$
where $\omega_{i,j}$ is the stress of Definition~\ref{defR}.
\end{remark}

\begin{definition}\label{non-par}
An equilibrium force-load $F$ on $G(P)$ is said to be {\it non-parallelizable} at vertex $p$,
if the following two conditions hold.

Suppose that the forces of $F$ at all edges adjacent to $p$ are $F_1,\ldots, F_s$.

\begin{itemize}
\item
Let $a_i\in \{0,1\}$ for $i=1,\ldots, s$. Then
$$
\sum\limits_{i=1}^s a_iF_i=0 \quad \hbox{implies} \quad a_1=\ldots=a_s.
$$

\item
All $2^{s-1}-1$ lines of forces defined by
$$
F_1+\sum\limits_{i=2}^s a_iF_i, \quad \hbox{where} \quad  (a_2,\ldots a_s)\in \{0,1\}^s\setminus \{ (1,\ldots,1) \}
$$
are distinct.
\end{itemize}
We sat that a tensegrity $(G(P),F)$ is {\it non-parallelizable} if it is non-parallelizable at every its vertex.
\end{definition}

\begin{remark}
If the first condition of non-parallelizability for $(G(P),F)$ is not fulfilled at $p$,
then either some edge can be deleted, or $p$ can be splitted in two vertices
(each edge is adjacent either to one copy of $p$ or to another). The new framework admits the same tensegrity.
So in some sense this tensegrity is realizable for a framework of a ``simpler'' graph
(i.e., such graph has a smaller first Betti number).
\end{remark}

\subsection{Geometry on lines and points}
In this subsection we introduce several geometric operations and relations on points and lines in $\r P^2$.
Geometric relations here have much in common to algebraic equations.

\subsubsection{Basic geometric operations on lines and points}
Let us start with basic four operations.

\vspace{1mm}

{\noindent
{\bf Operation I (2-point operation).} Denote by $\ell_1 \cap \ell_2$ the intersection point of two lines.
In the case of $\ell_1 = \ell_2$ we write $\ell_1 \cap \ell_2=true$.
Further we say
$$
\ell\cap true=true\cap \ell =true\cap true =true
$$
}

\vspace{1mm}

{\noindent
{\bf Operation II (2-line operation).} Denote by $(p_1,p_2)$ the line passing through $p_1$ and $p_2$.
In the case if the points coincide we write $(p_1,p_2)=true$.
Further we say
$$
(p, true)=(true,p)=(true,true) =true.
$$
In case when there is no ambiguity we write simply $p_1p_2$ instead of $(p_1,p_2)$.
}

\vspace{1mm}

{\noindent
{\bf Operation III (Generic point operation).} Given a line $\ell$, choose a point of $\ell\setminus S$,
where $S$ is a discrete set (In fact, for our constructions it is enough to consider two-point sets $S$).
}

\vspace{1mm}

{\noindent
{\bf Operation IV (Generic line operation).} Given a point $p$, choose a line passing through $p$,
and not contained in a certain finite collection of lines (actually we use only single line collections).
}

\begin{remark}\label{Cayley-remark}
Let us say a few words about Cayley algebra here.
Recall that Cayley algebra in symbolic computation (do not mess up with octonions)
is an algebra whose elements are subspaces of a given space. Cayley algebra has two operations:
the sum and the intersection.
Traditionally these two operations are denoted by standard logic operators AND and OR, respectively:
$$
\begin{array}{l}
L_1\vee L_2= L_1 + L_2;
\\
L_1\wedge L_2= L_1\cap L_2.
\end{array}
$$
In some texts they are referred as {\it join} and {\it meet} operations.

Notice that all lines and points in projective geometry are defined by some linear spaces
of dimensions 2 and 1 respectively.
Then Operations~I and~II are particular cases of join and meet operations respectively.
In order to avoid confusion with exterior products ``$\wedge$'' for exterior forms
we prefer to follow the set theoretic notation in this paper.
For more information on Cayley algebras we refer to~\cite{DRS}, \cite{WM}, and~\cite{HLi}.
\end{remark}

\begin{remark}
Let us observe in one example a range of points that are constructed by all the compositions of Operations~I and II.
Given the following four points of $\r P^2$:
$$
P_0=\big((0:0:1),(1:0:1), (0:1:1), (1:1:1)\big).
$$
Using Operations~I and~II one can construct any rational point $\q P^2$, i.e., any point
$$
(a:b:c), \quad \hbox{where $a,b,c$ are rational numbers, and $(a,b,c)\ne (0,0,0)$.}
$$
All the other points can not be constructed from four points of $P_0$.

Notice also that for any 4-tuple of points in general position there exists
a projective transformation that takes these 4-tuple to $P_0$.
\end{remark}

\subsubsection{Geometric relations}
Further we introduce the following three geometric relations.

\vspace{1mm}

{\noindent
{\bf 3-line relation}: We write $\ell_1 \cap \ell_2 \cap \ell_3=true$ if three lines
$\ell_1$, $\ell_2$, and $\ell_3$ has a common point.
We also say
$$
\ell_1 \cap \ell_2 \cap true =
\ell_1 \cap true \cap \ell_3 =
true \cap \ell_2 \cap \ell_3 =
true,
$$
(where $\ell_1$, $\ell_2$, $\ell_3$ are allowed to be ``{\it true}'' as well).
}

\vspace{1mm}

{\noindent
{\bf 3-point relation}: We write $(p_1,p_2,p_3)=true$ if three points $p_1$, $p_2$, and $p_3$
are in a line.
In addition, we set
$$
(p_1,p_2,true)=
(p_1,true,p_3)=
(true,p_2,p_3)=
true,
$$
(where $p_1$, $p_2$, $p_3$ can be ``{\it true}'').
}

\vspace{1mm}

{\noindent
{\bf Point-line relation}: We write $p\in\ell=true$ if $p$ is a point of $\ell$.
Here we set again
$$
p\in true=
true\in\ell=true,
$$
(where $\ell$ and $p$ can be ``{\it true}'').
}

\subsubsection{Geometric relations on configuration spaces of points and lines}

In what follows we consider the configuration spaces of lines and points of special types.
All points of such configuration spaces are fixed and
each line of the configuration space is not fixed but passes through a prescribed fixed point.

\begin{definition}
Consider an arbitrary configuration space $\Xi$ of fixed points $P$ and non-fixed lines $L$ passing through fixed points.
We say that a composition of several geometric operations and one geometric relation
on vertices and lines of $\Xi$ is a {\it geometric condition} on $\Xi$.
\end{definition}

\begin{definition}
We say that a system of geometric conditions on $\Xi$ is {\it fulfilled}
at $P$ if there exists a choice of non-fixed lines in $L$
such that every geometric condition in the system results in ``{\it true}'' for this choice of non-fixed lines.
\\
We say that two systems of geometric conditions on $\Xi$ are {\it equivalent} if for every configuration $P$
either these systems are both fulfilled or they are both not fulfilled.
\end{definition}

\begin{remark}
In fact, if we arrive at ``$true$'' at some iteration of a geometric condition calculation
then the geometric condition is fulfilled itself.
\end{remark}

\begin{example}
Let $\Xi=\big((p_1,\ldots, p_9),(\ell)\big|p_5\in\ell\big)$.
Then the following configuration
$$
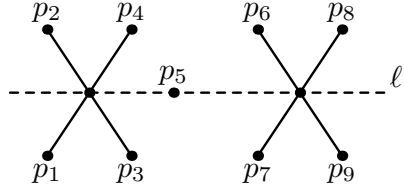
\epsfbox{pic.17}
$$
fulfills the system of geometric condition
$$
\left\{
\begin{array}{l}
p_1p_4\cap p_2p_3\cap \ell=true\\
p_6p_9\cap p_7p_8\cap \ell=true
\end{array}
\right.
.
$$
See also Examples~\ref{exx1}, \ref{exx2}, and~\ref{exx3} below.
\end{example}

\subsection{Frameworks in general position}\label{fimgp}

First we give the definition of a collection of lines in general position.

\begin{definition}
An $n$-tuple of lines in the projective plane are said to be {\it in general position} if
the union of their pairwise intersections is discrete and contains precisely $\frac{n(n-1)}{2}$ distinct points.
\end{definition}

We say that a graph is a {\it cycle} if it is homeomorphic to a circle.

\begin{definition}
Let $C(P)$ be a realization of a cycle $C$ in the projective plane, where $P=(p_1,\ldots,p_n)$.
We say that $C(P)$ is in general position if the lines passing through the edges of $C(P)$ are in general position.
\end{definition}

{\noindent
{\it Remark.}
Note that every edge of a cycle in general position has distinct endpoints. Otherwise the lines passing through
the edges have less distinct intersection points.
}

\vspace{2mm}

{\noindent
{\it Remark.}
In what follows we do not consider graphs having vertices of degree 1 or 2.
Vertices of degree 1 can be removed, while vertices of degree 2 together with adjacent edges can be
replaced by the edge connecting the endpoints of the adjacent edges distinct to the vertex.
This operation does not change the geometry of solutions.
In addition we do not consider non-connected graphs.
}

\vspace{2mm}

A {\it simple cycle} in a graph $G$ is a subgraph of $G$ homeomorphic to the circle.
(Here we use the word ``simple'' to avoid the confusion with closed walks also known as cycles in graphs.)
Let us now define a framework in general position.

\begin{definition}\label{Definition_general_position_framework}
Let $G$ be a connected graph on $n$ vertices whose vertices are all of degree at least 3.
We sat that $G(P)$ is a framework {\it in general position}
if every simple cycle of at most $n-1$ vertex is in general position.
\end{definition}

\subsection{Formulation of main result}
Let $G$ be a graph on $n$ vertices, and $G(P)$ be a framework with distinct points $P$.
Denote by $\Xi_G(P)$ the configuration space of the following points and lines:
\\
--- enumerated fixed points: all points of $P=(p_1,\ldots,p_n)$;
\\
--- enumerated non-fixed lines: at each point $p_i\in P$ we consider $\deg p_i-3$ lines passing through $p_i$.
We denote them by $\ell_{i,1},\ldots,\ell_{i,\deg p_i-3}$.

Notice that the dimension of $\Xi_G(P)$ is
$$
\sum\limits_{i=1}^n \big(\deg (p_i)-3\big),
$$
For instance, if all vertices are of degree 3, then the configuration space $\Xi_G(P)$
contains only fixed points and hence $\Xi_G(P)$ consists of 1 configuration.


One of the main results of the current paper is as follows.

\begin{theorem}\label{tensegrity=condi}
A framework $G(P)$ in general position admits a non-parallelizable tensegrity
if and only if
this framework satisfies the system of geometric conditions on $\Xi_G(P)$ for all simple cycles
of $G$.
\end{theorem}

\begin{remark}
The proof of this theorem is constructive, one can write the system of conditions explicitly (see Subsection~\ref{Quant_geom}).
The system of geometric conditions on $\Xi_G(P)$ would be defined much later (in Definition~\ref{trtrt}).
Until that, for simplicity, we consider this theorem as the existence statement.
\end{remark}

\begin{example}
For the second 6-vertex graph as on Figure~\ref{pic2.4} on the left we have the following geometric condition:

$$
\Big((p_1,p_2) {\cap} (p_4,p_5),
(p_2,p_3) {\cap} (p_5,p_6),
(p_3,p_4) {\cap} (p_1,p_6)\Big) =0.
$$
This geometric configuration is shown on Figure~\ref{pic2.4} (center).
Due to Pascal's theorem this geometric conditions is equivalent to the following one: {\it the points $p_1,\ldots, p_6$
are on a conic} (see Figure~\ref{pic2.4} on the right).
\begin{figure}
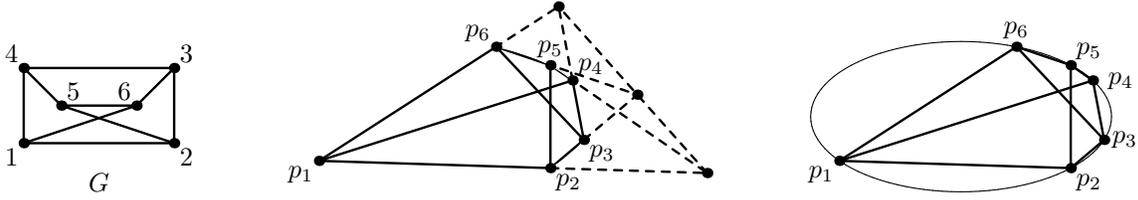

$$\epsfbox{pic2.4} \qquad\quad \epsfbox{pic2.6}\qquad\quad \epsfbox{pic2.5}$$
\caption{The graph $G$ (on the left), corresponding geometric condition $G$
(in the middle) and equivalent algebraic condition (on the right).}
\end{figure}\label{pic2.4}
\end{example}

\begin{remark}
In the proof of Theorem~\ref{tensegrity=condi} we use a technique of quantizations
and resolution schemes.
We give necessary notions, definitions and formulate several related propositions in Sections~3--5.
We return back to the proof of Theorem~\ref{tensegrity=condi} in Subsection~\ref{ProofSubSect}.
\end{remark}

\subsection{$\hf$-surgeries on tensegrities}

Consider a graph $G_H$ on 6 vertices $v_1,\ldots, v_6$ with edges
$$
v_1v_2, \quad
v_1v_3, \quad
v_1v_4, \quad
v_2v_5, \quad \hbox{and} \quad
v_2v_6.
$$
Denote also $G_\Phi$ the graph on vertices $v_1', v_2',v_3,v_4,v_5,v_6$ with edges
$$
v_1'v_2', \quad
v_1'v_3, \quad
v_1'v_5, \quad
v_2'v_4, \quad \hbox{and} \quad
v_2'v_6.
$$
\begin{definition}\label{HPhi}
Let $G$ be an arbitrary graph and let $G_H$ be a subgraph in $G$.
Let $G(P)$ be a framework on $G$ and let $G_H(Q)\subset G(P)$ has vertices $q_1,\ldots, q_6$.
Suppose that

\begin{itemize}
\item the vertices $q_1$ an $q_2$ has degree 3 in $G$;

\item $q_1\ne q_2$;

\item $q_1q_3 \ne q_2q_5$;

\item $q_1q_4 \ne q_2q_6$.
\end{itemize}
Consider $G_\Phi(Q')$ on points $q_1',q_2',q_3,q_4,q_5,q_6$ where
$$
q_1'=q_1q_3\cap q_2q_5\quad \hbox{and} \quad q_2'=q_1q_4\cap q_2q_6.
$$
We say that operation of replacing the subframework $G_H(Q)$ with $G_\Phi(Q')$ on the framework $G(P)$
is an {\it $\hf$-surgery} on $G(P)$ at the edge $q_1q_2$.
(See Figure~\ref{pic.3}.)
\end{definition}

\begin{figure}
$$\epsfbox{pic.3}$$
\caption{$\hf$-surgery.}\label{pic.3}
\end{figure}

\begin{proposition}\label{HQ-proposition}
Consider a pair of frameworks $G_H(Q)$ and $G_\Phi(Q')$ as above.
Let us given 4 forces

--- $F_{1}$ along $q_1q_3=q_1'q_3$;

--- $F_{2}$ along $q_1q_4=q_2'q_4$;

--- $F_{3}$ along $q_2q_5=q_1'q_5$;

--- $F_{4}$ along $q_2q_6=q_2'q_6$.
\\
Then the following two statements are equivalent
\begin{itemize}
\item there exists a force $F$ with line of force $q_1q_2$ satisfying
$$
F_{1}+F_{2}+F=F_3+F_4-F=0;
$$
\item there exists a force $F'$ with line of force $q'_1q'_2$ satisfying
$$
F_{1}+F_{3}+F'=F_2+F_4-F'=0.
$$
\end{itemize}
$($It is clear if the forces $F$ and $F'$ exist, they are uniquely defined.$)$
\end{proposition}

\begin{proof}
Suppose that such $F$ exists. Then denote
$$
F_{1',2'}=F_{1}+F_{3}, \quad \hbox{and} \quad F_{2',1'}=F_{2}+F_{4}.
$$
Let us prove that $F_{1',2'}+F_{2',1'}=0$.
We have
$$
\begin{aligned}
F_{1',2'}+F_{2',1'}&=(F_{1}+F_{3})+(F_{2}+F_{4})=
(F_{1}+F_{3})+(F_{2}+F_{4})+(F-F)\\
&=(F_{1}+F_{3}+F)+(F_{2}+F_{4}-F)=0+0=0.
\end{aligned}
$$
Therefore, first the line of force for $F_{1',2'}$ contains $q_2'$.
Secondly, $F_{2',1'}=-F_{1',2'}$.
Then $F'=F_{1',2'}$ is the stress at $q_1'q_2'$ satisfying
$$
F_{1}+F_{3}+F'=F_2+F_4-F'=0.
$$

The converse statement is true by the same reasons.
\end{proof}

\begin{corollary}
An $\hf$-surgery on a framework does not change the dimension of the space of equilibrium force-loads.
\qed
\end{corollary}

\begin{remark}
$\hf$-surgeries are projective analogs of surgeries of type II from~\cite{DKS}.
\end{remark}

\section{Monodromies of framed cycles and equilibrium force-loads}

In this section we study static properties of framed cycles.
We start in Subsection~3.1 with basic definitions of shift operators and monodromy operators.
Further in Subsection~3.2 we define equilibrium force-loads for framed cycles.
In Subsection~3.3 we study the projection operation on framed cycles (then we use them in the proof of the next subsection).
Finally, in Subsection~3.4 we prove that the existence of nonzero equilibrium force-loads on a framed cycle in general
position is equivalent to triviality of monodromy operators for this cycle.

\subsection{Framed cycles in general position, shift operators, monodromies}

We start with the notion of framed cycles in general position.
Further we introduce the notions of shift operators and monodromy operators.

\subsubsection{Framed cycles in general position}

\begin{definition}
We say that a realization of a cycle $C(P)$ with $P=(p_1,\ldots, p_k)$ in the projective plane has a {\it framing} if every  vertex $p_i$
is equipped with a line $\ell_i$ passing through it.
The framework $C(P)$ together with its framing is called the {\it framed cycle}.
We denote
$$
C(P,L)=\big((p_1,\ldots,p_k),(\ell_1,\ldots,\ell_k)\big).
$$
\end{definition}

\begin{definition}\label{fram-gp}
A framed cycle $C(P,L)$ is in {\it general position} if

--- the cycle $C(P)$ is in general position;

--- for all admissible $i$ we have: the line $\ell_i$ does not contain the points $p_{i-1}$ and $p_{i+1}$.

\end{definition}

\subsubsection{Shift operators}

Let us introduce shift maps for the lines of a framed cycle $C(P,L)$ along simple paths.

First of all we define a shift operator for consecutive framed lines.
Assume that $C(P,L)$ has $k$ vertices.
Below in this subsection we consider all indices mod $k$.
Let $p_{i}$ and $p_{i+1}$ be two consecutive points of the cycle, and let $\ell$ be any line
not passing through $p_{i}$ and $p_{i+1}$.

Define
$$
\xi_\ell[p_ip_{i+1},\ell_i,\ell_{i+1}]:\ell_i \to  \ell_{i+1},
$$
such that for every $p \in \ell_i$ we have
$$
\xi_\ell[p_ip_{i+1},\ell_i,\ell_{i+1}](p)=  \ell_{i+1}\cap\big((p_ip_{i+1}\cap \ell),p\big)
$$
(see Figure~\ref{pic.18}).

\begin{figure}
$$\epsfbox{pic.18}$$
\caption{A shift operation $\xi_\ell[p_ip_{i+1},\ell_i,\ell_{i+1}]$.
Here
$$
q=\xi_\ell[p_ip_{i+1},\ell_i,\ell_{i+1}](p)= \ell_{i+1}\cap\big((p_ip_{i+1}\cap \ell),p\big).
$$
}\label{pic.18}
\end{figure}

\begin{remark}
One might consider the line $\ell$ as the line at infinity for some affine chart.
In this affine chart the line
$\big((p_ip_{i+1}\cap \ell),p\big)$
is simply the line through $p$ parallel to $p_ip_{i+1}$.
\end{remark}

Now we define the shift operator for a path $p_i\ldots p_{i+s}$
in a framed cycle $C(P,L)$ in general position (summation of indices is mod $k$ as above).
Let also $\ell$ be a line that does not contain vertices of the path.
Set
$$
\begin{array}{l}
\xi_\ell[p_i\ldots p_{i+s};\ell_{i},\ldots, \ell_{i+s}]=\\
\qquad
\xi_\ell[p_{i+s-1}p_{i+s};\ell_{i+s-1},\ell_{i+s}]\circ \dots \circ\xi_\ell[p_ip_{i+1};\ell_{i},\ell_{i+1}]
:\ell_i \to  \ell_{i+s}.
\end{array}
$$

\subsubsection{Monodromy operators}

Consider a framed cycle $C(P,L)$ in general position, where $P=(p_1, \ldots, p_k)$
and $L=(\ell_1, \ldots, \ell_k)$.

\begin{definition}
Let $\ell$ be a line in $\r P^2$ which does not pass through vertices of the framed cycle $C(P,L)$.
The {\it monodromy} of $C(P,L)$ at  $\ell_i\in L$ is the operator
$$
\xi_\ell[p_ip_{i+1}p_{i+2}\ldots, p_{i-1}p_i;\ell_i,\ell_{i+1},\ell_{i+2},\ldots, \ell_{i-1},\ell_i]:\ell_i\to \ell_i.
$$
We denote it by $M_\ell(\ell_i,C(P,L))$.
\end{definition}

It is clear that the monodromy operator acts as a linear operator on $\ell_i$ with the origin at $p_i$.

\begin{definition}
A monodromy $M_\ell(\ell_i,C(P,L))$ is said to be {\it trivial} if
it is an identity map on $\ell_i$.
\end{definition}

\begin{proposition}\label{dif-mono}
Suppose that there exists $i$ such that $M_\ell(\ell_i,C(P,L))$ is trivial.
Then $M_\ell(\ell_j,C(P,L))$ is trivial for all $j\in\{1,\ldots, k\}$.
\end{proposition}

\begin{proof}
This statement follows directly from the definition of the monodromy.
\end{proof}

\begin{proposition}
The property of a monodromy to be trivial is a projective invariant.
\end{proposition}

\begin{proof}
The property of a monodromy to be trivial is written in terms of Operations~I--III,
applied to $(P,L)$,
which is a projective invariant property.
\end{proof}

Later in Corollary~\ref{monodromy-invariance} we show that the property of the monodromy to be trivial
does not depend on the choice of the line $\ell$.

\begin{definition}\label{mcc}
We say that a framed cycle in general position satisfies the {\it monodromy cycle condition} if it has a trivial monodromy.
\end{definition}

\subsection{Equilibrium force-loads on framed cycles}

Let us define equilibrium force-loads for framed cycles.

\begin{definition}\label{almost--eq}
Let $C(P,L)$ be a framed cycle in general position.
\begin{itemize}
\item{A {\it force-load} $F$ on a framed cycle $C(P,L)$ is an assignment of

--- stresses $F_{i,i+1}=-F_{i+1,i}$ for every edge $p_ip_{i+1}$ where $1\le i,j\le k$;

--- framing forces $F_i$ (for $1\le i \le k$) whose lines of forces are $\ell_i$ respectively.
}
\vspace{2mm}

\item{A force-load $F$ is called an {\it equilibrium} force-load on $C(P,L)$ if, in addition, the
following equilibrium condition is fulfilled at every vertex $p_i$ (here index addition is mod~$k$):
$$
F_{i,i-1}+F_{i,i+1}+F_i=0.
$$
}
\vspace{2mm}

\item{A force-load $F$ is called an {\it almost equilibrium} force-load on $C(P,L)$ if the
equilibrium condition is fulfilled at every vertex of $C(P,L)$ except one.
}
\end{itemize}
\end{definition}

\begin{remark}
Let us consider the cycle $C(P)$ as a rigid body and let $F$ be its equilibrium force-load.
Then the total action of all forces $F_i$ lines leaves the rigid body $C(P)$ at rest.
This is equivalent to the fact
$$
F_1+F_2+\ldots+F_k=0
$$
in the projective statics.
\end{remark}

\begin{proposition}\label{ef-l}
Consider a framed cycle $C(P,L)$ in general position.
Let $C(P,L)$ admits a nonzero equilibrium force load
then every almost equilibrium force load on this cycle is an equilibrium force-load.
\end{proposition}

\begin{proof}
Let $F$ be a nonzero equilibrium force-load on $C(P,L)$, and let $\hat F$ be an almost equilibrium force load on $C(P,L)$.
Since the equilibrium conditions fulfilled simultaneously for $F$ and $F'$ at all vertices of the cycle except one,
the force-load $F$ is proportional to the force-load $F'$. Hence $F'$ is an equilibrium force-load.
\end{proof}

\subsection{Projection operations on framed cycles and their properties}\label{Projection_operation}
In this subsection we introduce the notion of operations on framed cycles, which is
the main tool of induction in the proof of Proposition~\ref{projection-property} below.
Further we study their basic properties.

\subsubsection{Projection operations on framed cycles}
Consider a framed cycle $C(P,L)$ with $k\ge 4$ vertices.
Let also $i\in \{1,2,\ldots, k\}$.
Denote
$$
\begin{array}{l}
p_i'=p_{i-1}p_i\cap p_{i+1}p_{i+2};
\\
\ell_i'=p_i'(\ell_i\cap \ell_{i+1}).
\end{array}
$$
(Here as before, we set $p_0=p_k$, $p_{k+1}=p_1$, and $p_{k+2}=p_2$.)

Set
$$
\omega_i(C(P,L))= C(P',L'),
$$
where
\begin{equation}\label{e1}
\begin{array}{l}
P'=(p_1,p_2,\ldots,p_{i-1},p_i',p_{i+2},\ldots, p_{k}),
\quad \hbox{and}
\\
L'=(\ell_1,\ell_2,\ldots,\ell_{i-1},\ell_i',\ell_{i+2},\ldots, \ell_{k}).
\end{array}
\end{equation}

\begin{figure}
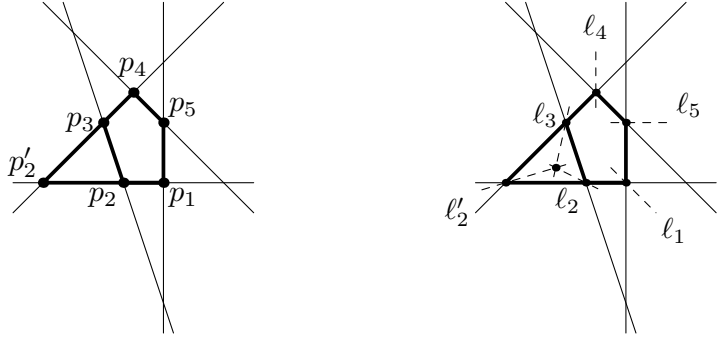

$$\epsfbox{pic.1}\qquad\qquad\qquad \epsfbox{pic.2}$$
\caption{A projection operation $\omega_2$.
(The point $p_2'=p_{1}p_2\cap p_{3}p_{4}$ is on the left, and
the line $\ell_2'=(p_2',\ell_2\cap \ell_{3})$ is on the right.)
}\label{pic.1}
\end{figure}

On Figure~\ref{pic.1} we consider an example of a projection
operation $\omega_2$ applied to the framed cycle
$C\big((p_1,\ldots p_5),(\ell_1,\ldots, \ell_5)\big)$.
We show the points $p_1,\ldots p_5,p_2'$ on the left,
and the lines $\ell_1,\ldots, \ell_5, \ell_2'$ on the right.

\subsubsection{Basic properties of projection operators}
Let us collect together several basic properties of projective operations $\omega_i(C(P,L))$.

\begin{proposition}\label{projection-property}
Let $C(P,L)$ be a framed cycle  in general position on $k\ge 4$ vertices and let $\omega_i$ be one of its projections.
Then we have
\\
$($i$)$ The cycle $\omega_i(C(P,L))$ is in general position.
\\
$($ii$)$ Let $j\notin \{i, i+1\}$ and let $\ell$ be a line that contains neither vertices of $C(P,L)$ nor $p_i'$.
Then we have
$$
M_\ell\big(\ell_j,C(P,L)\big)=M_\ell\big(\ell_j,\omega_i(C(P,L))\big).
$$
\\
$($iii$)$ The existence of a nonzero equilibrium force-load for $C(P,L)$ is equivalent to
the existence of a nonzero equilibrium force-load for $\omega_i(C(P,L))$.
\end{proposition}

\begin{proof} Let $\omega_i (C(P,L))=C(P',L')$, here we follow the notation of~(\ref{e1}) above.

{
\noindent
{\it Item $($i$)$}. Since $C(P,L)$ is in general position, the cycle $C(P)$ is in general position.
The set of lines through all edges in $C(P')$ coincide with the set of lines through edges in $C(P')$ minus
the line $p_ip_{i+1}$.
Hence the number of intersection points of lines through all edges in $C(P')$ is
$$
\frac{n(n-1)}{2}-n-1=\frac{(n-1)(n-2)}{2},
$$
So these lines are in general position. Therefore, $C(P')$ is in general position.
}

\vspace{2mm}

Recall that the only new line in the framing $L'$ is the line
$$
\ell_i'=(p_i',\ell_i\cap \ell_{i+1})
$$
through the point $p_i'=p_{i-1}p_i\cap p_{i+1}p_{i+2}.$
Denote $B=\ell_i \cap \ell_{i+1}$.

First, let us show that the line $\ell_i'$ does not contain $p_{i-1}$.
Since the line $\ell_{i+1}$ does not contain $p_{i}$, we have $B\ne p_i$.
Further since the line $\ell_{i}$ does not contain $p_{i-1}$, the point $B$ is not in the line of the edge $p_{i-1}p_i$.
Therefore $\ell_i'$ does not contain the edge $p_{i-1}p_i$.
Finally, since $C(P)$ is in general position, we have $p_i'=\ell_i'\cap p_{i-1}p_i\ne p_{i-1}$
and therefore the point $p_{i-1}$ is not in $\ell_i'$.

Secondly, by the same reasons $\ell_i'$ does not contain $p_{i+2}$.
Finally, all the other genericity conditions for the other lines of $C(P',L')$ are as
for the lines of the framed cycle $C(P,L)$.
Hence all genericity conditions are fulfilled. Therefore, the cycle $C(P',L')$ is in general position.

\vspace{2mm}
{
\noindent
{\it Item $($ii$)$}.
Consider two operators sending $\ell_{i-1}$ to $\ell_{i+2}$:
$$
\xi_\ell[p_{i-1}p_ip_{i+1}p_{i+2};\ell_{i-1},\ell_i,\ell_{i+1},\ell_{i+2}],
\quad \hbox{and} \quad
\xi_\ell[p_{i-1}p_i'p_{i+2};\ell_{i-1},\ell_i',\ell_{i+2}].
$$
}
Let us prove that they coincide.

Denote
$$
q_1=\ell\cap (p_{i-1},p_i),
\quad
q_2=\ell\cap (p_{i},p_{i+1}),
\quad
q_3=\ell\cap (p_{i+1},p_{i+2}).
$$

Let $q$ be a point of $\ell_{i-1}$.
Set
$$
A_1=\ell_i \cap qq_1,\quad
A_2=\ell_i' \cap qq_1 ,\quad
A_3=\ell_{i+1}\cap A_2q_3,\quad
\hbox{and}
\quad
$$
Recall also that $B=\ell_i'\cap\ell_i=\ell_i'\cap\ell_{i+1}.$

Without loss of generality we consider the affine chart with the line $\ell$ at infinity.
Now $A_1A_2$ is parallel to $p_{i-1}p_{i}$ and
$A_2A_3$ is parallel to $p_{i+1}p_{i+2}$.
Let us prove that
$A_1A_3$ is parallel to $p_{i}p_{i+1}$.
(All the points and lines of the affine chart are shown on Figure~\ref{pic.5}.
Note that the points $q_1$, $q_2$, and $q_3$ are at the line $\ell$ at infinity.
hence they are not in the affine chart.)

\begin{figure}
$$\epsfbox{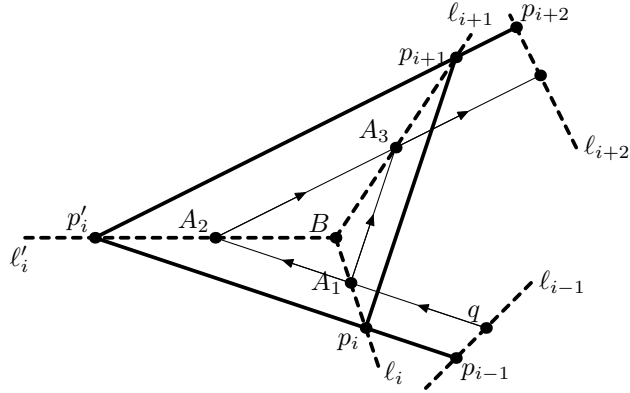}$$
\caption{Projection operation does not change monodromy operators, since
$
\xi_\ell[p_{i-1}p_i'p_{i+2};\ell_{i-1},\ell_i',\ell_{i+2}](q)
=\xi[p_{i-1}p_ip_{i+1}p_{i+2};\ell_{i-1},\ell_i,\ell_{i+1},\ell_{i+2}](q)$
for every $q\in \ell_{i-1}$.
}\label{pic.5}
\end{figure}

The triangle $p_i'p_iB$ is homothetic to the triangle $A_2A_1B$, and the coefficient of homothety is $\frac{|p_i'B|}{|A_2B|}$.
The triangle $p_i'p_{i+1}B$ is homothetic to the triangle $A_2A_3B$, and the coefficient of homothety is $\frac{|p_i'B|}{|A_2B|}$.
Hence the quadrangle $p_ip_i'p_{i+1}B$ is homothetic to the quadrangle $A_1A_2A_3B$, and the coefficient of homothety is $\frac{|p_i'B|}{|A_2B|}$.
Therefore, $p_ip_{i+1}$ is parallel to $A_1A_3$.
Hence we have
$$
\begin{aligned}
\xi_\ell[p_{i-1}p_i'p_{i+2};\ell_{i-1},\ell_i',\ell_{i+2}](q)
&=
\xi_\ell[p_i'p_{i+2};\ell_i'\ell_{i+2}]\circ\xi_\ell[p_{i-1}p_i';\ell_{i-1}\ell_i'](q)\\
&=
\xi[p_{i+1}p_{i+2};\ell_{i+1},\ell_{i+2}]{\circ}
\xi[p_ip_{i+1};\ell_i,\ell_{i+1}]{\circ}
\xi[p_{i-1}p_i;\ell_{i-1},\ell_i](q)
\\
&=\xi[p_{i-1}p_ip_{i+1}p_{i+2};\ell_{i-1},\ell_i,\ell_{i+1},\ell_{i+2}](q)
\end{aligned}
$$
Therefore,
$$
\xi_\ell[p_{i-1}p_i'p_{i+2};\ell_{i-1},\ell_i',\ell_{i+2}]=\xi[p_{i-1}p_ip_{i+1}p_{i+2};\ell_{i-1},\ell_i,\ell_{i+1},\ell_{i+2}].
$$

Since all the rest shift operators defining the monodromy $M_\ell\big(\ell_j,C(P,L)\big)$
are invariant under the projection $\omega_i$, we have
$$
M_\ell\big(\ell_j,C(P,L)\big)=M_\ell\big(\ell_j,\omega_i(C(P,L))\big).
$$

\vspace{2mm}
{
\noindent
{\it Item $($iii$)$}. Here the cycle $C(P',L')$ is obtained from $C(P,L)$ via an $\hf$-surgery
sending a subgraph $G_H(P)$ of the framework $G(P)$  on vertices
$(p_i,p_{i+1},p_{i-1},A_1,p_{i+2},A_3)$
to the graph $G_\Phi(P)$ on vertices $(p_i',B,p_{i-1},A_1,p_{i+2},A_3)$.
(Here $A_1$ and $A_3$ are some points on $\ell_i$ and $\ell_{i+1}$, and
$B=\ell_i\cap\ell_{i+1}$, see also Figure~\ref{pic.5}).
Now the statement follows directly from Proposition~\ref{HQ-proposition}.
}
\end{proof}

\subsection{Monodromy condition for a nonzero equilibrium force-load}

Let us formulate a necessary and sufficient condition of
the existence of a nonzero equilibrium force-load for a given framed cycle.

\begin{theorem}\label{force=monodromy}
Let $C(P,L)$ be a framed cycle in general position;
let $\ell$ be a line that does not contain intersection points of any pair of edges for $C(P,L)$;
and let $1\le i\le k$.
Then the monodromy operator $M_\ell(\ell_i,C(P,L))$ is trivial
if and only if
there exists a nonzero equilibrium force-load for $C(P,L)$.
\end{theorem}

We start with the following lemma.

\begin{lemma}\label{triangular}
Consider a triangular cycle $C(P,L)$ in general position, and let $\ell$ be a line that does not contain the vertices of $C(P,L)$.
Then the following three statements are equivalent:

$($a$)$ the lines $\ell_1$, $\ell_2$, $\ell_3$ meet in a point;

$($b$)$ there exists a nonzero equilibrium force-load for $C(P,L)$;

$($c$)$ the monodromy operator $M_\ell(\ell_i,C(P,L))$ is trivial for every $i\in\{1,2,3\}$.
\end{lemma}

\begin{proof} $($a$){\iff}($b$)$.
The equivalence of the first and the third statements is a classical statement.
Suppose that there exists a nonzero equilibrium force-load $F$ on $C(P,L)$.
Hence
$$
F_1+F_2+F_3=0, \quad \hbox{or, equivalently,} \quad  F_1=-F_2-F_3.
$$
The intersection point of force lines $F_2$ and $F_3$ belongs to the force line $F_2+F_3$ and hence to $F_1$.
Therefore, $\ell_1$, $\ell_2$, and $\ell_3$ intersect in a common point.

Conversely, let $\ell_1$, $\ell_2$, and $\ell_3$ meet in a point $B$. Consider
$$
F_i=a_idp_i\wedge dB \quad \hbox{for $i=1,2,3$}
$$
with nonzero real numbers $a_1,a_2,a_3$ such that $F_1+F_2+F_3=0$ which is equivalent to
\begin{equation}\label{eq2}
(a_1dp_1+a_2dp_2+a_3dp_3)\wedge dB=0.
\end{equation}
Since $p_1$, $p_2$, and $p_3$ are not in a line, we have $a_1dp_1+a_2dp_2+a_3dp_3\ne 0$.
Therefore, Equation~(\ref{eq2}) implies
$$
dB=\alpha(a_1dp_1+a_2dp_2+a_3dp_3)
$$
for some nonzero $\alpha$.
Set
$$
F_{i,j}=-\alpha a_ia_j dp_i\wedge dp_j.
$$
Then at every edge we have
$$
F_{i,i+1}+F_{i+1,i}=-\alpha a_ia_{i+1} dp_i\wedge dp_{i+1}-\alpha a_ia_{i+1} dp_{i+1}\wedge dp_i=0,
$$
and at every vertex we have
$$
\begin{aligned}
F_i+F_{i,i-1}+F_{i,i+1}&=
a_idp_i\wedge ({-}\alpha a_{i-1}dp_{i-1}{-}\alpha a_{i+1}dp_{i+1}+dB)
\\
&=\alpha a_idp_i\wedge \big({-}a_{i-1}dp_{i-1}{-}a_{i+1}dp_{i+1}+(a_{i-1}dp_{i-1}{+}a_{i}dp_{i}{+}a_{i+1}dp_{i+1})\big)\\
&=\alpha a_idp_i\wedge a_{i}dp_{i}
\\
&=0.
\end{aligned}
$$
Hence $F$ is a nonzero equilibrium force-load on $C(P,L)$.

\vspace{2mm}

\begin{figure}
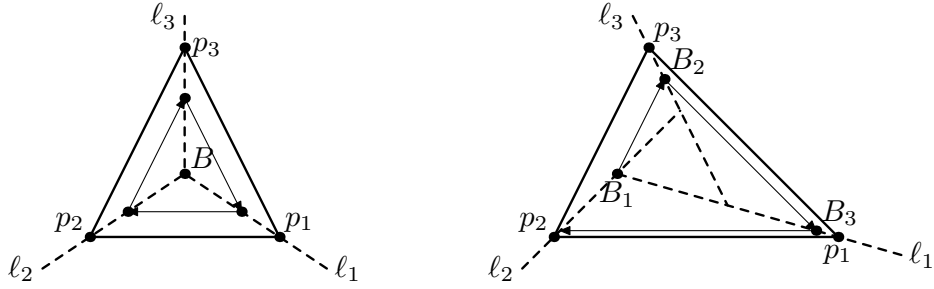

$$\epsfbox{pic.19}\qquad\qquad\epsfbox{pic.20}$$
\caption{Framed triangular cycles with trivial monodromy operators (on the left) and with
non-trivial monodromy operators (on the right).}\label{pic.19}
\end{figure}

{\noindent
$($a$){\iff}($c$)$.
Denote $B=\ell_1\cap \ell_2$.
Suppose $B \in \ell_3$, then $B$ is a fixed point for every monodromy (see Figure~\ref{pic.19}, on the left).
Therefore, all monodromies are trivial.
}

Suppose now, that $\ell_3$ does not contain $B$ (see Figure~\ref{pic.19}, on the right). Denote
$$
B_1=B, \qquad
B_2=\xi_\ell[p_2p_3;\ell_2,\ell_3](B_1), \quad \hbox{and} \quad
B_3=\xi_\ell[p_3p_1;\ell_3,\ell_1](B_3).
$$
Since $\ell_3 \ne \ell_1$, the point $B_2$ is not in $\ell_1$.
Further, since $p_1 \notin \ell$ we have
$$
p_3p_1\cap\ell\ne p_3p_2\cap\ell.
$$
Hence
$$
B_3=\xi_\ell[p_3p_1;\ell_3,\ell_1](B_2)\ne\xi_\ell[p_3p_2;\ell_3,\ell_2](B_2)=B_1.
$$
Therefore, $M_\ell(\ell_1,C(P,L))$ is not trivial.
Then by Proposition~\ref{dif-mono} all monodromies are not trivial.
\end{proof}

{
\noindent
{\it Proof of Threorem~\ref{force=monodromy}.}
Let us prove the statement of the theorem by induction in the number of vertices in the cycle.
}

\vspace{1mm}
{
\noindent
{\it Base of induction.} If the cycle $C(P,L)$ is triangular then the statement of Theorem~\ref{force=monodromy}
follows from Lemma~\ref{triangular}.
}

\vspace{1mm}
{
\noindent
{\it Step of induction.} Let the statement holds for every framed cycle in general position on $n$ vertices.
Let us prove the statement for an arbitrary framed cycle in general position on $n{+}1$ vertices.
}

Let $C(P,L)$ be a framed cycle in general position on $n{+}1$ vertices, and let $\omega_i$ be one of its projection operations.
Then from the one hand, by Proposition~\ref{projection-property}({\it ii}) the monodromy at every edge of $C(P,L)$ is trivial
if and only if the monodromy at the corresponding edge of the projection $\omega_i(C(P,L))$ is trivial.
From the other hand by Proposition~\ref{projection-property}({\it iii})
the existence of a nonzero equilibrium force-load for $C(P,L)$ is equivalent
to the existence of a nonzero equilibrium force-load for $\omega_i(C(P,L))$.

Now the statement of the theorem follows directly from the induction assumption,
since the framed cycle $\omega_i(C(P,L))$ is a framed cycle in general position on $n$ vertices
(see Proposition~\ref{projection-property}({\it i})).
\qed

\begin{corollary}\label{monodromy-invariance}
Let $C(P,L)$ be a framed cycle in general position
and let $\ell'$ and $\ell''$ be a pair of lines
that do not contain intersection points of every pair of distinct edges in the cycle.
Then
$M_{\ell'}(\ell_i,C(P,L))$ is trivial if and only if $M_{\ell''}(\ell_i,C(P,L))$ is trivial.
\end{corollary}

\begin{proof}
Suppose that $M_{\ell'}(\ell_i,C(P,L))$ is trivial.
Hence by Theorem~\ref{force=monodromy} there exists an equilibrium force-load on the framed cycle $C(P,L)$.
Therefore, by Theorem~\ref{force=monodromy} the monodromy $M_{\ell''}(\ell_i,C(P,L))$ is trivial.
The converse is similar.
\end{proof}


\section{Resolution schemes at vertices}

In this section we describe a technique of resolution schemes at vertices.
It will be further used to generate framed cycles of quantizations in the next section.
In Subsection~4.1 we introduce resolution schemes and define force-loads at them.
Further in Subsection~4.2 we formulate the notions of weakly and strongly generic resolution schemes.
We define $\hf$-surgeries for resolutions schemes and the corresponding equivalent relation in Subsection~4.3.
Finally, in Subsection~4.4 we prove finiteness of equivalent strongly generic resolution schemes
(and also provide their actual number).

\subsection{Resolution schemes and force-loads on them}

We say that an edge of a tree is a {\it leaf} if one of its vertices is of degree 1.
All other edges are said to be {\it interior}.

Denote the set of all lines in the projective plane by $\gr$.
Note that the set $\gr$ is naturally isomorphic to the Grassmannian of 2-dimensional planes in $\r^3$,
i.e., to $\hbox{\rm Gr}(2,3)$.

\begin{definition}
Consider an unrooted full binary tree $T$ (i.e., the degree of every vertex of $T$ is either 1 or 3).
Let
$$
\L: E(T) \to \gr.
$$
We say that a pair $(T,\L)$ is a {\it resolution scheme}
at point $p\in \r P^2$ if for every edge $e\in T$ we have $p\in\L(e)$.
Denote it by $(T,\L)_p$.
\end{definition}

\begin{definition}
Let $(T,\L)_p$ be a resolution scheme. Assume that $T$ has vertices $v_1,\ldots, v_k$.

\begin{itemize}
\item{Consider an edge $v_iv_j$ of $T$. A {\it stress} on $v_iv_j$ of the resolution scheme $(T,\L)_p$
is a pair of forces  $(F_{i,j},F_{j,i})$ satisfying:

--- the line of force $F_{i,j}$ coincides with the line $\L (v_iv_j)$;

--- $F_{i,j}+F_{j,i}=0$.
}

\item{A {\it force-load} $F$ on the resolution scheme $(T,\L)_p$ is an assignment of stresses at every edge of $T$.
In addition we say that  $F_{i,j}=0$ if $v_iv_j$ is not an edge of $T$.
}

\item{A force-load $F$ is said to be an {\it equilibrium} force-load if
at every vertex $p_i$ of degree~3
the following equilibrium condition is fulfilled:
$$
\sum\limits_{\{j|j\ne i\}} F_{i,j}=0.
$$
}
\end{itemize}
\end{definition}

\begin{remark}
In some sense an equilibrium force-load on a resolution scheme can be seen as
a part of an infinitesimally perturbed tensegrity at point $p$.
\end{remark}

\subsection{Resolution schemes in general position}

In this subsection we study two types resolution schemes in general position.

\subsubsection{Weakly generic resolution schemes}
Let us start with weakly generic resolution schemes.

\begin{definition}
We say that a resolution scheme $(T,\L)_p$ is {\it weakly generic} if
for every pair of adjacent edges $(v_iv_j,v_jv_s)$ of $T$ we have $\L(v_iv_j)\ne \L(v_jv_s)$.
\end{definition}

\begin{proposition}\label{weakly-generic}
Let $(T,\L)_p$ be a weakly generic resolution scheme, then

$($i$)$ $(T,\L)_p$ has a nonzero equilibrium force-load;

$($ii$)$ if an equilibrium force-load is nonzero at some edge, then it is nonzero at every edge;

$($iii$)$ every two equilibrium force-loads on $(T,\L)_p$ are proportional.
\end{proposition}

\begin{proof}
Let us construct an equilibrium force-load on $(T,\L)_p$ starting from an arbitrary edge $e$.
Fix an arbitrary nonzero stress at the edge $e$.
Let us consider all edges adjacent to $e$.
The equilibrium conditions at vertices of degree 3 uniquely define the stresses at all adjacent edges.
Inductively attaching adjacent vectors further we uniquely extend the collection of constructed stresses at edges to the
equilibrium force-load on the entire tree $T$.
This can be done for the entire tree $T$, since it has only vertices of degree 1 and of degree 3.
This concludes the proof of Item~({\it i}).

\vspace{1mm}

At each step of the induction discussed above we obtain nonzero stresses
at edges of $T$, hence the statement of Item ({\it ii}) holds.

\vspace{1mm}

Finally, the linear combination of two equilibrium force-loads is an equilibrium force-load.
Therefore, Items ({\it ii}) and ({\it iii}) are equivalent.
\end{proof}

\subsubsection{Strongly generic resolution schemes}

Let us give a more restrictive definition of resolution schemes in general position.

\begin{definition}
Let $(T,\L)_p$ be a weakly generic resolution scheme and let
$F$ be a nonzero equilibrium force-load at it.
Suppose that the forces of $F$ at all the leaves of $T$ are $F_1,\ldots, F_s$.
We say that $(T,\L)_p$ is {\it strongly generic} if the following two conditions hold.
\begin{itemize}
\item
Let $a_i\in \{0,1\}$ for $i=1,\ldots, s$. Then
$$
\sum\limits_{i=1}^s a_iF_i=0 \quad \hbox{implies} \quad a_1=\ldots=a_s.
$$

\item
All $2^{s-1}-1$ lines of forces defined by
$$
F_1+\sum\limits_{i=2}^s a_iF_i, \quad \hbox{where} \quad  (a_2,\ldots a_s)\in \{0,1\}^{s-1}\setminus \{ (1,\ldots,1) \}
$$
are distinct.
\end{itemize}
\end{definition}

\begin{remark}
It is clear that the lines forces
$$
F_1+\sum\limits_{i=2}^s a_iF_i \quad \hbox{and} \quad \sum\limits_{i=2}^s (1-a_i)F_i
\quad \hbox{where} \quad  (a_2,\ldots a_s)\in \{0,1\}^{s-1}\setminus \{ (1,\ldots,1) \}.
$$
are always negative to each other and hence their lines of forces coincide.
\end{remark}

\begin{example}
Suppose that the lines for the leaves of a resolution scheme with $s$-leaves are distinct.
Then if $s=3$ then the resolution scheme is strongly generic.
Further if $s=4,5$ then the resolution scheme is strongly generic if and only if
all the lines of the framing for the resolution scheme are pairwise distinct.
For $s>5$ we have more complicated conditions on strongly generic resolution schemes.
\end{example}

\subsubsection{On nonzero equilibrium force-loads at weakly generic resolution schemes}

Let us observe the following static property of lines at interior edges of
weakly generic resolutions schemes.

\begin{proposition}\label{force-structure}
Let $(T,\L)_p$ be a weakly generic resolution scheme and let $e$ be an interior edge.
Let $\{e_1,\ldots,e_r\}$ be the subset of all leaves of $T$ that are in one of the connected components for $T\setminus\{e\}$.
Consider a nonzero equilibrium force-load $F$ on $(T,\L)_p$,
and let $F_1,\ldots,F_r$ be the forces acting along the edges $e_1,\ldots,e_r$.
Then the line of force for
$$
F_1+\ldots+F_r\ne 0
$$
coincides with $\L(e)$.
\end{proposition}

\begin{proof}
The proposition is a direct corollary of the law of force addition.
\end{proof}

In particular, Proposition~\ref{force-structure} implies the following result.

\begin{corollary}\label{cor-force}
A nonzero equilibrium force-load on leaves and the type of a tree $T$ uniquely define the resolution scheme.
\qed
\end{corollary}

\subsection{Equivalent resolution schemes}

Our next goal is to introduce equivalence relation on strongly generic resolution schemes.

First of all we define $\hf$-surgeries for strongly generic resolution schemes.
In the next definition we consider graphs
$G_H$ with vertices $v_1,\ldots, v_6$, and $G_\Phi$ with vertices
$v_1', v_2',v_3,v_4,v_5, v_6$ as in Definition~\ref{HPhi} (see also Figure~\ref{pic.4}).

\begin{definition}\label{surII-rs}
Let $(T,\L)_p$ be a strongly generic resolution scheme.
Consider $G_H\subset T$ and $G_\Phi$ such that $v_1'$ and $v_2'$ are not vertices of $T$.
Consider a nonzero force-load $F$ on $(T,\L)_p$.
Set $\ell$ as a line of force for $F_{3,1}+F_{2,5}$.

We say that operation of replacing the subgraph $G_H$ with $G_\Phi$ and changing $\L$ to $\L'$ defined as
$$
\begin{array}{ll}
\L'(v_1'v_2')=\ell; & \quad \L'(v_1'v_3)= \L(v_1,v_3); \\
\L'(v_1'v_5)= \L(v_2,v_5); & \quad \L'(v_2'v_4)= \L(v_1,v_4); \\
\L'(v_2'v_6)= \L(v_2,v_6); & \quad \L'(e)=\L(e)\hbox{ for any other edge $e$}
\end{array}
$$
is an {\it $\hf$-surgery} on $(T,\L)_p$ at the interior edge $v_1v_2$.
(See Figure~\ref{pic.4}.)
\end{definition}

\begin{figure}
$$\epsfbox{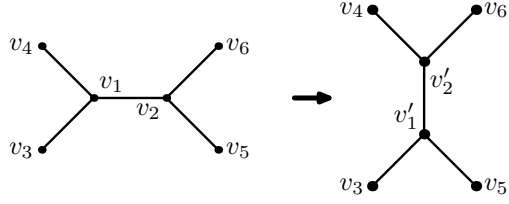}$$
\caption{$\hf$-surgery at an interior edge $v_1v_2$ of a resolutions scheme.}\label{pic.4}
\end{figure}

Later is Subsubsection~\ref{Geom-framed} we show how to define the surgery using Operations~I--IV.

\begin{remark}
Notice that $\hf$-surgery is not defined if $F_{3,1}+F_{2,5}=0$.
\end{remark}

For strongly generic resolution schemes we have the following proposition.

\begin{proposition}\label{HF-and-forceloads}
Let $(T,\L)_p$ be a strongly generic resolution scheme.
Then for every $G_H\subset T$ the following three statements hold.

{
\noindent
{\it $($i$)$} The corresponding $\hf$-surgery is well-defined.
}

{
\noindent
{\it $($ii$)$} The $\hf$-surgery does not change equilibrium force-loads
for all the edges of $(T,\L)_p$ except for the edge $v_1v_2$ where the surgery takes place.
In particular all equilibrium force-loads at leaves are preserved by every $\hf$-surgery.
}

{
\noindent
{\it $($iii$)$} The resulting resolution scheme is strongly generic.
}
\end{proposition}

\begin{proof}
{\it Item $($i$)$.}
By Proposition~\ref{force-structure},
since $(T,\L)_p$ is strongly generic,
the lines at edges of $(G_H,\L)_p$ are defined by forces with distinct lines of forces.
Hence $(G_H,\L)_p$ is strongly generic, in particular, the lines for edges of the framework for $(G_H,\L)_p$ do not coincide with each other.
This implies that $F_{3,1}+F_{2,5}\ne 0$ and thus $\ell'_{1,2}$ is uniquely defined.
Therefore, $\hf$-surgery is well-defined.

\vspace{1mm}

{
\noindent
{\it Item $($ii$)$} follows directly from the definition of $\hf$-surgery.
}

\vspace{1mm}

{
\noindent
{\it Item $($iii$)$.} The $\hf$-surgery does not change the forces on the leaves.
Hence the resulting resolution scheme is weakly generic, otherwise some of the sums at leaves are proportional.
Since the forces at leaves before and after the $\hf$-surgery coincide, two conditions of strong genericity are fulfilled.
Together with weak genericity this implies strong genericity of the resulting resolution scheme.
}
\end{proof}

\begin{definition}\label{eqqquiv}
We say that two strongly generic resolution schemes are {\it equivalent}
if
there exists a sequence of $\hf$-surgeries taking one of them to another.
\end{definition}

\subsection{Finiteness of equivalent resolution schemes}

Finiteness of equivalent resolution schemes follows directly from the following statement.

\begin{proposition}
Let $(T,\L)_p$ be a strongly generic resolution scheme $($where $T$ has $n$ leaves$)$.
Then the set of resolution schemes equivalent to $(T,\L)_p$ are in one to-one-correspondence with
the set of all unrooted binary full trees with $n$ marked leaves.
\end{proposition}

\begin{proof}
By Proposition~\ref{HF-and-forceloads}({\it ii}) the set of restrictions of equilibrium force-loads to the leaves are the same
for equivalent resolution schemes.
By Proposition~\ref{weakly-generic}({\it i}) there exists a nonzero equilibrium force-load $F$ on $(T,\L)_p$,
hence the set of restrictions contains nonzero elements.
By Corollary~\ref{cor-force} the equilibrium force-load on leaves and the type of the tree uniquely defines the resolution scheme.
Hence for every unrooted binary full tree $T'$ with $n$ leaves there exists at most one resolution scheme equivalent to $(T,\L)_p$.
We have injectivity.

By Proposition~\ref{HF-and-forceloads}({\it i}) the $\hf$-surgery at every interior edge of a strongly generic resolution scheme is well-defined,
and by  Proposition~\ref{HF-and-forceloads}({\it iii}) the image is also a strongly generic resolution scheme.
Hence for every unrooted binary full tree $T'$ with $n$ leaves there exists a resolution scheme $(T',\L')_p$ equivalent to $T(p,L)$.
This implies surjectivity.
\end{proof}

\begin{remark}
Note that the number of rooted binary full trees
with $n$ labeled leaves is precisely $(2n-1)!!$ (e.g., see ex.~5.2.6 in~\cite{Sta}).
Here the root may be considered as a vertex with additional leave with ``empty'' color.
Hence the number of unrooted binary full trees is $(2n-1)!!$ as well.
\end{remark}

\section{Quantization of a graph}

The main goal of this section is to define quantizations of graphs
and prove a necessary and sufficient condition of existence of non-parallelizable tensegrities in terms of quantizations.
We start in Subsection~5.1 with the notion of graph quantizations.
Further in Subsection~5.2 we introduce framed cycles of frameworks related to quantizations.
We generalize the notion of tensegrity to the case of quantizations in Subsection~5.3.
In Subsection~5.4 we study force-loads for framed cycles of quantizations.
Finally, in Subsection~5.5 we formulate and prove a necessary and sufficient condition of existence of non-parallelizable tensegrities (Theorem~\ref{tensegrity=quantiza}).

\subsection{Definition of a quantization}

Let us start with the following general definition.

\begin{definition}
Consider a graph $G$ on vertices $v_1,\ldots, v_n$.
Let $\T$ be a disjoint collection of unrooted binary full trees $T_1,\ldots, T_n$ such that
the leaves of the tree $T_i$ are in one-to-one correspondence with
the edges adjacent to $p_i$ $($where $i=1,\ldots,n)$.
\begin{itemize}
\item Denote by $G_\T$ the connected graph obtained from $\T$ by gluing together pairs of leaves of trees
corresponding to the same edges of $G$. We say that $G_\T$ is a {\it resolution graph}
for $G$ with respect to $\T$.

\item The tree $T_i$ is a {\it resolution tree} for $G_\T$ at vertex $p_i$.

\item We say that an edge of $G_\T$ is a {\it leaf} if it is obtained by gluing leaves in $\T$. All
the other edges of $G_\T$ are said to be {\it interior} edges.
\end{itemize}
\end{definition}

Now we are ready to define quantizations.

\begin{definition}
Let $G$ be a graph, let $G(P)$ be a framework on $G$,
and let $\T$ be a collection of trees for all vertices of $G(P)$.
A {\it quantization} $G_\T(\L,P)$ is a pair $(G_\T,\L)$ where
$$
\L: E(G_\T) \to \gr
$$
such that the following two conditions are fulfilled.
\begin{itemize}
\item {\bf Leaf condition:}
Let $e\in G_\T$ be a leaf corresponding to the edge $p_ip_j\in G(P)$, hence
$$
\L(e)=p_ip_j.
$$

\item {\bf Interior edge condition:}
Let $e\in G_\T$ be an interior edge corresponding to the point $p_i\in G(P)$, hence
$$
p_i\in \L(e).
$$
\end{itemize}
\end{definition}

\begin{definition}
Let $G_\T(\L,P)$ be a quantization of $G$. Consider a point $p_i\in P$ and the corresponding tree $T_i\in\T$.
Here we consider $T_i$ as a tree naturally embedded to $G_\T$. Denote by $\L|_{T_i}$
the restriction of $\L$ to $T_i$.
We say that $(\L|_{T_i},T_i)_{p_i}$ is the {\it resolution scheme of the quantization $G_\T(\L,P)$ at vertex $p_i$}.
\end{definition}

Let us continue with a particular example.

\begin{example}
On Figure~\ref{pic22} we consider a framework $G(P)$, a collection of resolution schemes $\T$ and the corresponding quantization $G_\T(\L,P)$.
The values of the function $\L$ at edges of the resolution schemes and the quantization are shown on the edges.
\end{example}

\begin{figure}
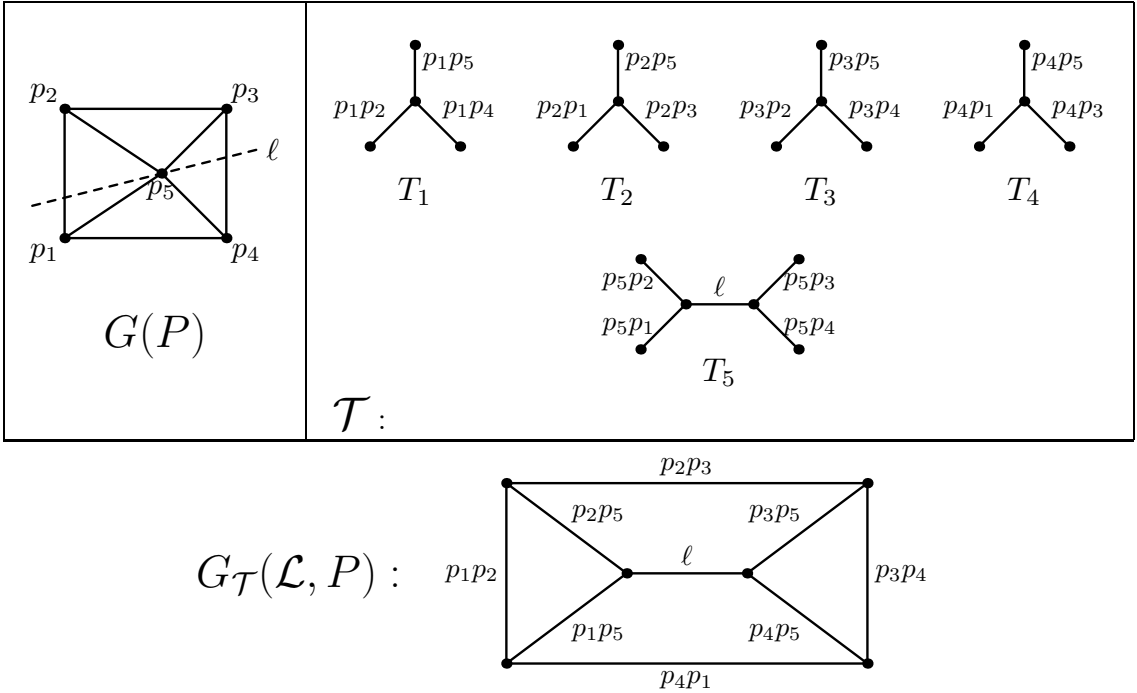

\begin{tabular}{|c|c|}
\hline
$
\begin{array}{c}
\epsfbox{pic.21}
\\
\quad
\\
{\hbox{\Large$G(P)$}}
\end{array}
$
&
$
\begin{array}{l}
\quad\\
\epsfbox{pic.22}
\\
{\hbox{\Large $\T$} }:
\\
\end{array}
$
\\
\hline
\end{tabular}
$$
{\hbox{\Large $G_\T(\L,P):$}} \quad \begin{array}{l}\epsfbox{pic.23}\end{array}
$$
\caption{A framework $G(P)$, a collection of resolution schemes $\T$ and the corresponding quantization $G_\T(\L,P)$.}
\end{figure}\label{pic22}

Further we will use the following notion of genericity for quantizations.

\begin{definition}
A quantization is called {\it generic} if all its resolutions schemes are strongly generic.
\end{definition}

\subsection{Framed cycles associated to quantizations}

First, we define the framing for two adjacent edges of frameworks.

\begin{definition}\label{Definition_fra}
Let $G(P)$ be a framework in general position and let $G_\T(\L,P)$ be its generic quantization.
Consider two edges $e_{ij}$ and $e_{ik}$ in $G(P)$ with a common vertex $p_i$, let also $(T_i,\L)_{p_i}$ be
the resolution scheme at $p_i$.
Consider a resolution scheme $(T_i',\L')_{p_i}$ equivalent to $(T_i,\L)_{p_i}$ such that
the leaves $\L'^{-1}(e_{ij})$ and $\L'^{-1}(e_{ik})$ are adjacent to the same vertex $v$ in $T'_i$,
and let $e$ be the third edge adjacent to $v$.
We say that the line $\L(e)$  is the {\it associated framing for the pair of edges $(e_{ij},e_{ik})$ at $p_i$} and
denote it by $\ell_{jik}$.
\end{definition}

The above definition leads to the natural notion of framed cycles associated to quantizations.

\begin{definition}\label{defff}
Let $G(P)$ be a framework in general position and let $G_\T(\L,P)$ be its generic quantization.
Consider a simple cycle $C=q_1\ldots q_s$ in $G(P)$ that does not contain all vertices of $P$.
Denote by $C(G,\T,P,\L)$ the framed cycle with consequent vertices $q_1\ldots q_s$ and framing
lines $\ell_{i{-}1,i,i{+}1}$ at $q_i$ (for $i=1,\ldots,s$). We say that this cycle
is a {\it framed cycle associated to quantizations}.
\end{definition}

We should mention the following property of the framings.

\begin{proposition}\label{prop_fra}
The framing $\ell_{jik}$ in Definition~\ref{Definition_fra}
does not depend on the choice of the equivalent resolution scheme $(T_i',\L')_{p_i}$ with adjacent leaves $\L'^{-1}(e_{ij})$ and $\L'^{-1}(e_{ik})$.
\end{proposition}

\begin{proof}
Let $(T_i',\L')_{p_i}$ and $(T_i'',\L'')_{p_i}$ be two equivalent resolution schemes satisfying the condition of the proposition.
Let also $e$ be an edge in $T_i'$ adjacent to $\L'^{-1}(e_{ij})$ and $\L'^{-1}(e_{ik})$.
Since $(T_i',\L')_{p_i}$ and $(T_i'',\L'')_{p_i}$ are equivalent, we can connect them by a sequence of ${\hf}$-surgeries.
Since $(T_i',\L')_{p_i}$ and $(T_i'',\L'')_{p_i}$ are strongly generic, we can choose theses surgeries avoiding the edge $e$.
Hence the line $\ell_{jik}$ is preserved by this sequence of $\hf$-surgeries (by Definition~\ref{HPhi}).
Therefore, the line $\ell_{jik}$ is uniquely defined.
\end{proof}

\begin{remark}
In order to construct a resolution scheme $(T_i',\L')_{p_i}$ starting from $(T_i,\L)_{p_i}$ we propose the following.
Let $v_1\ldots v_s$ be a simple path connecting the leaf $v_1v_2$ where $\L(v_1v_2)=p_ip_j$
and the leaf $v_{s-1}v_s$ with framing to $\L(v_{s-1}v_s)=p_ip_k$.
Then we consequently perform $s{-}3$ $\hf$-surgeries along the edges $v_2v_3$, $v_3v_4,\ldots, v_{s-2}v_{s-1}$.
As a result we have a resolution scheme $(T_i',\L')_{p_i}$ whose leaves $\L'^{-1}(e_{ij})$ and $\L'^{-1}(e_{ik})$ have a common vertex.
\end{remark}

\subsection{Tensegrities on quantizations}

Let us extend the definition of tensegrity to the case of quantizations.

\begin{definition}
Consider a quantization $G_\T(\L,P)$.

\begin{itemize}
\item{A {\it stress} on an edge $v_iv_j\in G_\T$ is an assignment of two forces $F_{i,j}$ and $F_{j,i}$
whose line of forces coincide with the line $\L(v_iv_j)$ and such that $F_{i,j}+F_{j,i}=0$.
}

\vspace{1mm}

\item{A {\it force-load} $F$ on a quantization is an assignment of stresses for all edges.
Additionally we set $F_{i,j}=0$ if $v_iv_j$ is not an edge of $G_\T$.
}

\vspace{1mm}

\item{A force-load $F$ is called an {\it equilibrium} force-load if, in addition, the
following equilibrium condition is fulfilled at every vertex
$v_i$:
$$
\sum\limits_{\{j|j\ne i\}} F_{i,j}=0.
$$
}

\item{A force-load $\hat F$  on $G(P)$ is called a force-load {\it induced} by $F$ on the quantization $G_\T(\L,P)$
if for every leaf $v_iv_j$ the stress on $v_iv_j$ for $F$ coincides with the stress on the corresponding edge of $G(P)$ for $\hat F$.

}

\end{itemize}
\end{definition}

We have the following natural property of the induced force-loads.

\begin{proposition}\label{quantiation-forceload-tensegrity}
If $F$ is an equilibrium force-load on a quantization $G_\T(\L,P)$
then the induced force-load $\hat F$ on the framework $G(P)$ is also an equilibrium force-load.
\end{proposition}

\begin{proof}
It is enough to check that at every vertex $p_i$ the forces of the framework sum up to zero.
This follows from the fact that the sum of all forces on the leaves of any equilibrium force-load on a tree is always zero.
\end{proof}

Further we define quantizations associated to non-parallelizable equilibrium force-loads on frameworks.

\begin{definition}\label{force-load-guant}
Let  $F$ be a  non-parallelizable equilibrium force-load $F$ on a framework $G(P)$
and let $G_\T$ be a resolution graph for $G$.
Consider
$$
\L_F: E(G_\T) \to \gr
$$
defined as follows.
Consider an edge $e\in G_\T$. Then there exists $i$ such that $e\in T_i$ for $\T_i \in \T$.
Let $F_{i_1,j_1},\ldots, F_{i_s,j_s}$ be the forces at all leaves of one of the connected components of $T_i\setminus e$
(we assume that the second indices for the forces correspond to 1-valent vertices).
Set $\L(e)$ as the line of force for
$$
F_{i_1,j_1}+\ldots+F_{i_s,j_s}.
$$
The quantization $G_\T(\L_F,P)$ is called the {\it quantization associated to  the pair} $(G(P),F)$.
\end{definition}

We have the following simple statements for the associated quantizations.

\begin{proposition}\label{quantiation-uniqueness}
$($i$)$ The quantization $G_\T(\L_F,P)$ is uniquely defined for every non-parallelizable equilibrium force-load $F$ on $G(P)$.
\\
$($ii$)$ The proportional  non-parallelizable equilibrium force-loads on $G(P)$ define the same quantization for a given $\T$.
\qed
\end{proposition}

\subsection{Framed cycles of quantizations and force-loads on them}
We start this subsection with several notions and definitions related to framed cycles in quantizations.
Further we prove one important static property of them.

\subsubsection{Basic definition}
Let us define cycle resolutions, their framings, and corresponding force-loads for the framed cycles.

\begin{definition}
Consider a graph $G$ with a resolution graph $G_\T$.
Let $C$ be a simple cycle in $G$. We say that a simple cycle in $G_\T$
containing all the leaves corresponding to the edges of $G$ is
the {\it cycle resolution} of $C$. We denote it by $C_\T$.
\end{definition}

\begin{definition}
Consider a cycle resolution $C_\T$ in a generic quantization $G_\T(\L,P)$.
A {\it framed cycle} $C_\T(G,\T,P,\L)$ is a cycle $C_\T$ whose each vertex, say $v_i$,
is equipped with the line $\L(v_iv_j)$, where $v_iv_j$ is the only edge in $E(G_\T)\setminus E(C_\T)$
adjacent to $v_i$.
\end{definition}

\begin{definition}\label{almost--eq}
Let $C_\T(G,\T,P,\L)$ be a framed cycle of a generic quantization $G_\T(\L,P)$.
\begin{itemize}
\item{A {\it force-load} $F$ on a framed cycle $C_\T(G,\T,P,\L)$ is an assignment of

--- stresses $F_{i,i+1}=-F_{i+1,i}$ for every its edge $v_iv_{i+1}$, whose lines of forces coincide
with $\L(v_iv_{i+1})$;

--- framing forces $F_i$ (for $1\le i \le k$), whose lines of forces coincide with the lines of the framing.
}
\vspace{2mm}

\item{A force-load $F$ is said to be an {\it equilibrium} force-load on $C_\T(G,\T,P,\L)$ if at every vertex $v_i$  we have:
$$
F_{i,i-1}+F_{i,i+1}+F_i=0.
$$
}

\item{A force-load $F$ is called an {\it almost equilibrium} force-load on $C_\T(G,\T,P,\L)$ if the
equilibrium condition is fulfilled at every vertex of $C_\T(G,\T,P,\L)$ except one.
}
\end{itemize}
\end{definition}

\subsubsection{Almost equilibrium force-loads for consistent framed cycles}

We continue with the definitions of consistency for framed cycles of frameworks and for frameworks themselves.

\begin{definition}\label{consistent}
Let $G(P)$ be a framework in general position and let $C$ be its cycle.
A generic quantization of a framework is said to be {\it consistent} at the cycle $C$
if the framed cycle of $C(G,\T,P,\L)$ satisfies the monodromy cycle condition (see Definition~\ref{mcc}).
\end{definition}

\begin{definition}\label{consistent2}
A generic quantization of a framework $G(P)$ in general position is said to be {\it consistent}
if it is consistent at each simple cycle that does not contain all the vertices of $P$.
\end{definition}

Consistent framed cycles of framework satisfy the following property.

\begin{proposition}\label{quant--normal}
Let $G(P)$ be a framework in general position and let $G_\T(\L,P)$ be a generic quantization of $G(P)$.
Consider a framed cycle $C(G,\T,P,\L)$ of $G(P)$.
Assume that $C$ is consistent.
Then any almost equilibrium force-load at $C_\T(G,\T,P,\L)$ is an equilibrium force-load.
\end{proposition}

\begin{proof}
Suppose that a cycle $C_\T(G,\T,P,\L)$ has interior edges of resolution schemes.
Then we remove such edges by applying $\hf$-surgeries at these edges.
By Proposition~\ref{HQ-proposition} $\hf$-surgeries do not change the property of a cycle to have an (almost) equilibrium
nonzero force-load.

Now the question is reduced to the framed cycle of $C_{\tilde T}(G,\tilde T,P,\L)$
whose edges are all leaves.
In this case equilibrium (almost equilibrium) force-loads for the cycle $C_{\tilde T}(G,\tilde T,P,\L)$
coincide with equilibrium (almost equilibrium) force-loads for the cycle
$$
C(G,\tilde T,P,\L)=C(G,T,P,\L)
$$
of $G(P)$.
By Theorem~\ref{force=monodromy}, since $C(G,T,P,\L)$ is consistent (i.e., it has a trivial monodromy),
$C(G,T,P,\L)$ has a nonzero equilibrium force-load. Now the statement of this proposition
follows directly from Proposition~\ref{ef-l}.
\end{proof}

\subsection{A necessary and sufficient condition for a tensegrity with nonzero stresses}

The following statement is the main ingredient for the proof of Theorem~\ref{tensegrity=condi}.

\begin{theorem}\label{tensegrity=quantiza}
A framework in a general position admits a non-parallelizable tensegrity
if and only if
there exists a consistent generic quantization of a graph.
\end{theorem}

\begin{remark}
The choice of the trees in the resolution schemes of the quantization does not change the set of frameworks
that admit nonzero tensegrities, although an appropriate choice can simplify the expressions for geometric conditions defining this set
(for further detailed see in Subsection~\ref{Quant_geom} below).
\end{remark}

\begin{remark}
The monodromy operators are multiplicative with respect to loop addition,
in particular the monodromies does not depend on how one add loops.
Hence the monodromies are well-defined at the elements of the first homology group of the graph $G$ (i.e., on $H_1(G)$).
\end{remark}

\begin{proof}
{\noindent
{\it Sufficient condition.}
Suppose that there exists a non-parallelizable tensegrity $(G(P),F)$ (note that this tensegrity
has nonzero stresses at all edges).
By Proposition~\ref{quantiation-uniqueness}({\it i}) the tensegrity $(G(P),F)$ uniquely defines the
quantization $G_\T(\L_F,P)$. Let us show that $G_\T(\L_F,P)$ is consistent.
}

\vspace{1mm}

First of all, let us show that $G_\T(\L_F,P)$ is generic.
This is equivalent to the fact that every resolution scheme $(T_i,\L)$ of this quantization is strongly generic.
The last directly follows from  non-parallelizability of the force $F$ at every point $p_i$.

Secondly, we show that every simple framed cycle $C(G,T,P,\L_F)$ associated to the quantization and not passing through
all vertices of $G$ is in general position.
Since $G(P)$ is in general position, and the cycle does not pass through all edges of $G$ the cycle $C(P)$ is in general position.
Now let us check the conditions for framed lines at vertices.
Let $p_{i-1}$, $p_i$, and $p_{i+1}$ be arbitrary consecutive vertices of the cycle $C(P)$
and let $\ell_{i-1,i,i+1}$ be the corresponding framing at $p_i$.
Then the direction of the line $\ell_{i-1,i,i+1}$ is defined by the sum of two forces representing edges $p_ip_{i-1}$ and $p_ip_{i+1}$.
Since $F$ is non-parallelizable, both forces are nonzero and the edges are non-parallel.
Hence the line $\ell_{i-1,i,i+1}$ contains neither $p_{i-1}$  nor $p_{i+1}$.
Therefore, by Definition~\ref{fram-gp} the cycle $C(G,T,P,\L_F)$ is in general position.

\vspace{1mm}

Assume that  a simple framed cycle $C(G,T,P,\L_F)$
does not pass through all vertices of $G$.
Let us prove that it has trivial monodromies.
Consider a force-load $\tilde F$ on the cycle $C(G,T,P,\L_F)$ defined as follows:
\begin{itemize}
\item The stresses at the edges of the cycle $C(G,T,P,\L_F)$ coincide with the stresses at the edges of the tensegrity $(G(P),F)$;

\item The force at the framed line at vertex $p_i$ coincides with the sum of all the forces at the adjacent edges
to $p_i$ in $G\setminus C$.
\end{itemize}
By Definition~\ref{force-load-guant}, the line of the framing at vertex $p_i$ coincides with the line of the sum of the forces at the adjacent edges at $p_i$,
hence $\tilde F$ is a force-load on $C(G,T,P,\L_F)$ for all admissible $i$.
Since $F$ is an equilibrium force-load for $G(P)$,
an equilibrium condition for $\tilde F$ at each vertex of $C(G,T,P,\L_F)$ is fulfilled.
Thus $\tilde F$ is a nonzero equilibrium force-load on $C(G,T,P,\L_F)$.
As we have shown above, the cycle $C(G,T,P,\L_F)$ is in general position.
Therefore, by Theorem~\ref{force=monodromy} for any generic line $\ell$ (as in Theorem~\ref{force=monodromy})
the monodromies of the cycle are trivial.
Hence by Definition~\ref{consistent} the cycle $C(G,T,P,\L_F)$ satisfies the monodromy cycle condition.

\vspace{1mm}

Therefore, $G_\T(\L_F,P)$ is consistent.

\vspace{2mm}

{\noindent
{\it Necessary condition.}
Suppose now that $G_\T(\L,P)$ is a consistent generic quantization of $G(P)$.
Let us construct an equilibrium force-load by induction on vertices of the resolution graph $G_\T$.
We construct a sequence
$$
\big((G_0,F), (G_1,F), \ldots, (G_N,F)\big),
$$
where $G_i$ is a collection of vertices and edges of $G_\T$
and $F$ is an assignment of forces to all edges of $G_i$.
In addition this collection satisfy all the following properties:
\begin{itemize}

\item we have $G_0 \subset G_1 \subset \ldots \subset G_N$;

\item for every $i<N-1$ the set $G_{i+1}$ has exactly one vertex more than $G_i$;

\item each vertex of $G_i$ is adjacent to exactly three edges of $G_i$;

\item each edge $e$ of $G_i$ is equipped with a nonzero stress whose direction coincides with $\L(e)$;

\item at every edge $e\in G_{i}\cap G_{j}$ the forces of the force-loads $(G_{i},F)$ and $(G_{j},F)$ coincide;

\item the forces at every vertex of $G_i$ sum up to zero;

\item the graph $G_0$ contains no vertices and one edge of $G_\T$;

\item $G_N=G_\T$.
\end{itemize}
}
If such sequence is constructed, $F$ is an equilibrium force-load for $G_\T(\L,P)$.
Then the induced force-load $\hat F$ is an equilibrium force-load for $G(P)$.

\vspace{2mm}

{\noindent{\it Base of induction.}
Let us start with any edge $v_{i_0}v_{j_0}$ of $G_\T(\L,P)$ with a nonzero stress $(F_{i_0,j_0},F_{j_0,i_0})$ on it
whose lines of force coincide with $\L(v_{i_0}v_{j_0})$. We set this as $(G_0,F)$.
}

\vspace{2mm}

{\noindent{\it Step of induction.}
It is important to choose a correct the order for adding new vertices.
Denote by $V'(G_\T)$ the set of all vertices of $G_T$ that are not adjacent to interior edges of a resolution tree $T_n$.
}

Suppose that we have already constructed the pair $(G_s,F)$ (we assume that $s< N$) let us construct $(G_{s+1},F)$.
If $V'(G_\T)$ is not a subset of vertices of  $G_s$
we choose a new vertex $v_i$ from $V'(G_\T)$ such that $v_i$ is adjacent to at least one of the edges
$v_iv_j\in G_s$.
If $V'(G_\T)$ is already a subset of vertices of  $G_s$ then we pick any
remaining vertex $v_i$ adjacent to at least one of the edges $v_iv_j\in G_s$.

Let $v_kv_i$ and $v_lv_i$ be the remaining two edges adjacent to $v_i$ in the quantization,
add them to $G_{s+1}(F)$ in case if they are not already in it.

We keep all the stresses at all the edges of $G_s\subset G_{s+1}$.
In particular the edge $v_iv_j$ has a stress $(F_{i,j},F_{j,i})$.
Define the stresses $(F'_{i,k},F'_{k,i})$ and $(F'_{i,l},F'_{l,i})$ at $v_iv_k$ and $v_iv_l$ respectively using the following two conditions

1) $F_{i,j}+F'_{i,k}+F'_{i,l}=0$;

2) $F'_{i,k}=-F'_{k,i}$ and $F'_{i,l}=-F'_{l,i}$;

3) the lines of forces for $F'_{i,k}$ and $F'_{i,l}$ coincide with the lines $\L(v_iv_k)$ and $\L(v_iv_l)$.
\\
It is clear that $F'_{i,k}$ and $F'_{i,l}$ are uniquely defined by these two conditions.
In case if $v_k$ or $v_l$ are not in $G_s$ we set respectively $F_{i,k}=F'_{i,k}$ or $F_{i,l}=F'_{i,l}$.

\vspace{1mm}

In case if $v_k$ is in $G_s$, we have already defined another stress while adding the vertex $v_k$.
In this case $G_{s+1}$ contains a framed cycle $C_\T$ with consecutive triple of vertices $v_jv_iv_k$.

If $V'(G_\T)$ is not a subset of vertices of $G_s$ then the framed cycle
$C(G,\T,P,\L)$ is consistent since it does not pass through all the vertices of $G(P)$.
We set
$$
\tilde C(G,\T,P,\L)=C(G,\T,P,\L).
$$

Suppose now $V'(G_\T)$ is already contained in $G_s$.
Then it might happen that the cycle $C\subset G$ corresponding to $C_\T\subset G_\T$ runs through all vertices of $G$.
Since all vertices of $G$ are at least of degree 3, we have at least $\frac{3n}{2}$ edges.
Therefore, if $n>3$ then there is an edge $e\ne q_kq_i$ which is not passing through $p_n$.
Using the edge $e$ one can make another simple cycle $\tilde C$ that has less than $n$ vertices and
such that $\tilde C_\T \cap T_n=C_\T \cap T_n$.
The obtained cycle $\tilde C$ is consistent, since $G_\T(\L,P)$ is consistent.

By construction, the force-load $F$ at $\tilde C_\T(G,\T,P,\L)$ is almost equilibrium.
Then by Proposition~\ref{quant--normal} the force-load $F$ at $\tilde C_\T(G,\T,P,\L)$ is an equilibrium force-load,
since cycle $\tilde C(G,\T,P,\L)$ is consistent.

Therefore, all forces of $G_{s+1}(F)$ sum up to zero at every vertex. All the other required conditions are fulfilled by construction.
Hence we are done with the step of induction.

\vspace{2mm}

As the output of the inductive process described before we get a nonzero equilibrium force-load $F$ at $G_N=G_\T$.
Since the quantization is generic, the resulting equilibrium force-load is non-parallelizable.
This concludes the proof of the necessary condition.
\end{proof}

\section{Geometric conditions on realizability of tensegrities}\label{geometry}

In this section we study algorithmic questions related to explicit construction of geometric conditions for tensegrities
and prove Theorem~\ref{tensegrity=condi}.

\subsection{Construction of geometric conditions for systems defining tensegrities}\label{Quant_geom}

In this subsection we discuss the main building blocks for writing down the geometric conditions.
We start with a natural correspondence between the the elements of $\Xi_G(P)$ and quantizations of $G(P)$.
Further we show how to construct all the lines in simple framed cycles starting from the lines of
the quantizations.
After that we give the algorithm to rewrite an $\hf$-surgeries on the resolution schemes in terms of Operations~I--IV.
Then we discuss how to express monodromy cycle conditions for framed cycles in general position
in terms of Operations~I--IV.
Finally, we define geometric conditions for the framed cycles associated to quantizations.
The results of this subsection are used in the next two subsections.

\subsubsection{A natural correspondence between the elements of $\Xi_G(P)$ and
the set of all quantizations for $G(P)$}\label{wwwww}

Consider a framework $G(P)$ and one of its quantization graphs $G_\T$.
For every point $p_i\in P$ we enumerate the interior edges of $G_\T$
corresponding to the resolution scheme $T_i$ as follows: $e_{i,1},\ldots, e_{i,\deg p_i-3}$.

Once enumeration of edges of $G_\T$ is fixed, the choice of $\L$ in a quantization
uniquely defines the element of $\Xi_G(P)$.
Namely we pick the non-fixed lines according to the rule
$$
\ell_{i,j}=\L(e_{i,j}),
$$
where $1\le i\le n$ and $1 \le j \le \deg p_i-3$.
It is clear that this correspondence is bijective.

\subsubsection{Construction of lines in the framed cycles associated to quantizations in terms of geometric operations}\label{Geom-framed}
Consider an $\hf$-surgery on a resolution scheme, here we follow the notation of Definition~\ref{surII-rs}.
Let us show how to construct the new line $\ell=\L'(v_1'v_2')$ using Operations~I--IV.

\vspace{2mm}

{
\noindent
{\it Step 1.} First of all, let us choose an affine chart.
Namely we pick of a point $p_\infty\ne p$ at a line (say at $\L(v_1v_2)$)
and a line $\ell_\infty\ne \L(v_1v_2)$ through $p_\infty$ (which is a combination of Operations~III and~IV).
}

\vspace{1mm}
{
\noindent
{\it Step 2.} Take an arbitrary point $p'\notin \{p,p_\infty\}$ in the affine chart on the line $\L(v_1v_3)$ (Operation~III).
}

\vspace{1mm}
{
\noindent
{\it Step 3.} Consider the line $\hat \ell$ passing through $p'$ and parallel to $\L(v_1v_2)$. Namely,
$$
\hat \ell=p(\L(v_1v_2)\cap\ell_\infty).
$$
This is a combination of Operations~I and~II.
}

\vspace{1mm}
{
\noindent
{\it Step 4.} Set
$$
p''=\hat\ell \cap \L(v_2v_5).
$$
Notice that $p''\ne p$. Here we use Operation~I.
}

\vspace{1mm}
{
\noindent
{\it Step 5.} Draw lines $\ell'$ and $\ell''$ through the points $p'$ and $p''$
parallel to $\L(v_1v_4)$ and $\L(v_2v_6)$ respectively:
$$
\begin{array}{l}
\ell'=p'(\L(v_1v_4)\cap\ell_\infty);\\
\ell''=p''(\L(v_2v_6)\cap\ell_\infty);\\
\end{array}
$$
Here we have used two Operations~I and two Operations~II.
}

\vspace{1mm}
{
\noindent
{\it Step 6.} Consider $p'''=\ell'\cap \ell''$ (Operation~I), note that by construction we have $p'''\ne p$.
}

\vspace{1mm}
{
\noindent
{\it Step 7.} Finally, we get $\ell=\L'(v_1'v_2')=pp'''$ (Operation~II).
}

\begin{figure}
$$\epsfbox{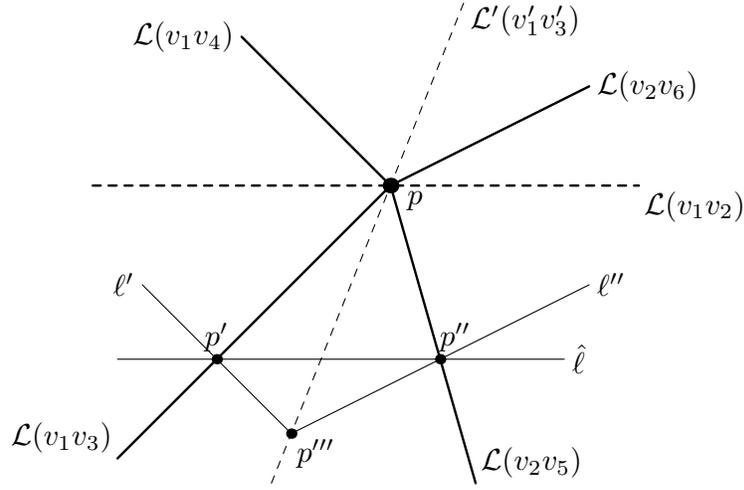}$$
\caption{Geometric construction of $\hf$-surgery on a resolution scheme.}\label{pic.13}
\end{figure}

\begin{definition}\label{abcdr}
Any line $\ell_{jik}$ in the framing associated to a quantization $G_\T(\L,P)$ is defined by a certain
composition of $\hf$-surgeries (see Definitions~\ref{Definition_fra} and~\ref{eqqquiv}), and hence
by the above it is defined by a composition of Operations~I--IV on the lines of $\L(E(G_\T))$.
For simplicity we fix one of the choices of sequences of Operations~I--IV defining the framing $\ell_{jik}$
and call it the {\it sequence of geometric operations} defining the line $\ell_{jik}$.
\end{definition}

\begin{remark}\label{Interesting}
It is interesting to observe that if the degree of a vertex is 3, then
the lines of the associated framing are defined by the edges.
Namely, let $p_i$ be a vertex of $G$ of degree 3. Assume that $p_i$ is adjacent to edges
$e_{ij}$, $e_{ik}$, and $e_{is}$. Then we have $e_{ik}\subset \ell_{jis}$,
$e_{ij}\subset \ell_{kis}$, and $e_{is}\subset \ell_{jik}$:
$$
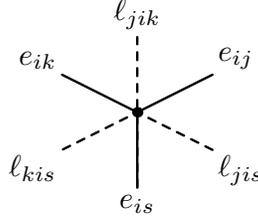
\epsfbox{pic.15}
$$
Another remark concerns a simple relation on the lines of the associated framing for vertices of degree 4.
Let $p_i$ be a vertex of $G$ of degree 4.
Assume $p_i$ is adjacent to edges $e_{ij}$, $e_{ik}$, $e_{is}$, and $e_{it}$.
Then $\ell_{jik}= \ell_{sit}$:
$$
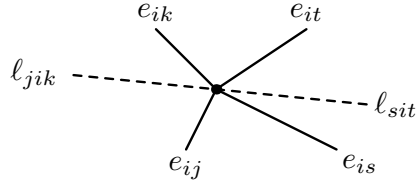
\epsfbox{pic.16}
$$
Similarly one has $\ell_{jit}\subset \ell_{kis}$ and $\ell_{jis}\subset \ell_{kit}$.
\end{remark}

\subsubsection{Construction of geometric conditions for framed cycles in general position}\label{cond-fcgp}
Let $C\big((p_1,\ldots,p_k),(\ell_1,\ldots,\ell_k)\big)$ be a framed cycle in general position.
Let us iteratively apply $k-3$ times the projection operation defined in Subsection~\ref{Projection_operation} replacing
the first and the second vertices by a new one, i.e.,
$$
\underbrace{\omega_1\circ\ldots\circ\omega_1}_{\hbox{$k-3$ times}}(C(P,L)).
$$
Each time the number of vertices in the cycle decreases by one.
We end up with the framed cycle $((p,p_{k-1},p_k),(\ell,\ell_{k-1},\ell_k)\big)$ in general position.

\begin{definition}\label{cycle_condd}
The condition
$$
\ell\cap\ell_{k-1}\cap\ell_k=true
$$
is the {\it geometric condition} defined by $C$.
\end{definition}

\begin{remark}
Here $\ell$ is the result of the composition of 2-point and 2-line operations on points $p_1,\ldots,p_k$ and lines
$\ell_1,\ldots,\ell_k$ arisen in the above composition of projection operations.
In fact, one can fix a different choice of the sequence of projective operation.
There are $\frac{k!}{3!}$ different possibilities to do so.
They all lead to equivalent geometric conditions.
\end{remark}

\begin{proposition}\label{zzzz}
Let $C(P,L)$ be a framed cycle in general position.
Then $C(P,L)$ satisfies the monodromy cycle condition
if and only if it satisfy any geometric condition defined by this cycle.
\end{proposition}

\begin{proof}
The statement for triangular cycles follows directly from  Lemma~\ref{triangular}.
The statement for framed cycles with 4 and more vertices is reduced to the triangular statement
by projection operations.
Equivalence of the initial and the reduced statements follows from Proposition~\ref{projection-property}
and the definition of the projection operation in terms of 2-point and 2-line operations (see Subsection~\ref{Projection_operation}).
\end{proof}

\begin{example}
If a cycle has three vertices then the relation is
$$
\ell_1\cap\ell_2\cap \ell_3=true,
$$
the sketch is as follows:
$$
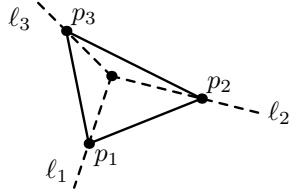
\epsfbox{pic.10}
$$
In case of four vertex cycles we have:
$$
(\ell_1\cap\ell_4,\ell_2\cap\ell_3,p_1p_2\cap p_3p_4)=true,
$$
the sketch is as follows:
$$
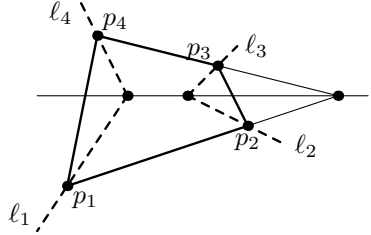
\epsfbox{pic.11}
$$
Finally, for five vertex cycles we have:
$$
(\ell_2\cap\ell_3,p_1p_2\cap p_3p_4)\cap \ell_1\cap(\ell_4\cap\ell_5,p_1p_3\cap p_3p_4)=true,
$$
the sketch is as follows:
$$
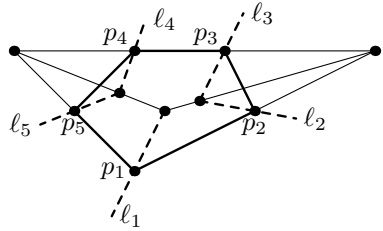
\epsfbox{pic.12}$$
Notice that if the number of vertices is greater than 3, then there are several different choices of equivalent geometric conditions.
\end{example}

\subsubsection{Geometric conditions for framed cycles associated to quantizations}

Finally, combining together the above two constructions we have the following descriptive definition.

\begin{definition}\label{trtrt}
Let $G(P)$ be a framework,
let $G_\T(\L,P)$ be a generic quantization for $G(P)$,
and let $C$ be a simple cycle of $G(P)$ that does not contain all the vertices of $P$.
Consider the geometric condition for the cycle $C(G,\T,P,\L)$ as in Definition~\ref{cycle_condd}.
The framed lines of the cycle $C(G,\T,P,\L)$ are written in terms of
points of $P$ and lines of $\L(E(G_\T))$ in Definition~\ref{abcdr}.
By Subsubsection~\ref{wwwww} the lines of $\L(E(G_\T))$ are identified with the lines of of $\Xi_G(P)$.
So we end up with a geometric condition on the points and the lines of $\Xi_G(P)$.
We say that this condition is the {\it geometric condition on $\Xi_G(P)$} for the framed cycle $C(G,\T,P,\L)$
in the quantization $G_\T(\L,P)$.
\end{definition}

\begin{remark}
In fact the geometric condition on $\Xi_G(P)$ depends neither on the choice of configuration $P$ nor on the choice of lines in the configuration.
Although it is defined using quantizations, it is a characteristic of the graph $G$ itself.
\end{remark}

Let us prove the following property for geometric conditions on $\Xi_G(P)$.

\begin{corollary}\label{cor_geommm}
Let $G(P)$ be a framework in general position,
let $G_\T(\L,P)$ be a generic quantization for $G(P)$,
and let $C$ be a simple cycle of $G(P)$ that does not contain all the vertices of $P$.
Then the quantization $G_\T(\L,P)$ is consistent at $C$ if and only if
the quantization satisfies the geometric condition on $\Xi_G(P)$ for the cycle $C(G,\T,P,\L)$
with non-fixed lines identified with the lines associated to quantization.
\end{corollary}

\begin{proof}
By Definition~\ref{consistent} the cycle $C$ is consistent if and only if
$C(G,\T,P,\L)$ satisfies the monodromy cycle condition.
Since $G(P)$ is a framework in general position and $G_\T(\L,P)$ is its generic quantization,
the framed cycle $C(G,\T,P,\L)$ is in general position.
By Proposition~\ref{zzzz} the cycle $C(G,\T,P,\L)$ satisfies the monodromy cycle condition
if and only if it satisfies the geometric condition defined by this cycle.
Hence, by Definition~\ref{trtrt}
the last is equivalent to the fact that the cycle $C(G,\T,P,\L)$ satisfies the geometric condition on $\Xi_G(P)$.
\end{proof}

\begin{example}
Let us consider a simple example related to the Desargues configuration as in Example~\ref{SectionDesargues}
(see Figure~\ref{pic2.1}, on the left).
Consider the following graph and a cycle $p_2p_3p_6$ in it.
\begin{figure}
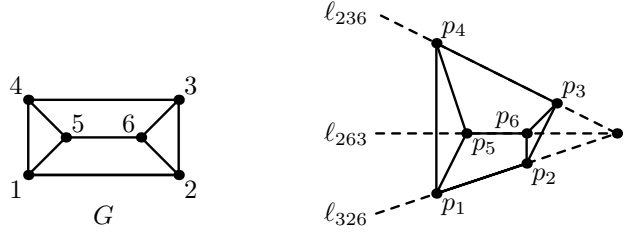

$$\epsfbox{pic2.1} \qquad\qquad \epsfbox{pic2.7}$$
\caption{The graph $G$ (on the left) and the consistency condition for the cycle $p_2p_3p_6$ (on the right).}\label{pic2.1}
\end{figure}
Here the cycle $p_2p_3p_6$ is consistent if
$$
\ell_{326}\cap\ell_{263}\cap\ell_{236}=true,
$$
see Figure~\ref{pic2.1} (on the right).
According to Remark~\ref{Interesting} the line $\ell_{326}=p_1p_2$;
$\ell_{263}=p_5p_6$; and $\ell_{236}=p_3p_4$.
Hence the above condition coincides with
$$
p_1p_2\cap p_5p_6 \cap p_3p_4=true.
$$
\end{example}

\subsection{Proof of Theorem~\ref{tensegrity=condi}}\label{ProofSubSect}
By Theorem~\ref{tensegrity=quantiza} a framework $G(P)$ in general position admits a non-parallelizable tensegrity
if and only if there exists a consistent quantization of a graph.
By Definition~\ref{consistent2} a quantization is consistent if
and only if it is consistent at each simple cycle that does not contain all vertices of $G$.
By Corollary~\ref{cor_geommm} a generic quantization is consistent at a simple cycle if and only if
in satisfy the geometric condition on $\Xi_G(P)$ for this cycle
with non-fixed lines coinciding with the lines associated to quantization.
This concludes the proof.
\qed

\begin{remark}
Quantization of a graph is an auxiliary tool to get the geometric conditions.
Once the geometric conditions are constructed, the quantization is no longer needed.
\end{remark}

\subsection{Techniques to construct geometric conditions defining tensegrities}\label{techniques}

In this subsection we give the summary of the algorithm
to write down the system of geometric conditions for the existence
of non-parallelizable tensegrities.

\vspace{2mm}

{\noindent
{\bf Data.} We start with a framework $G(P)$ in general position.
}

\vspace{1mm}

{\noindent
{\it Step 1.} Fix a collection of resolution schemes $\T$ and the corresponding configuration space $\Xi_G$.
Here we consider a family of quantizations $G_\T(\L,P)$ with fixed $G_\T$ and $P$,
and with the lines of $\L$ acting as parameters.
}

\vspace{1mm}

{\noindent
{\it Step 2.} Pick all simple cycles $C_1,\ldots, C_N$ in $G$ that does not pass through all the points of $G$.
}

\vspace{1mm}

{\noindent
{\it Step 3.} Using the algorithm described in Subsubsection~\ref{Geom-framed}
we write all lines $\ell_{jik}$ in terms of compositions of Operations~I--IV
on the points of $P$ and the lines of $\L(E(G))$. (See Definition~\ref{abcdr}.)
}

\vspace{1mm}

{\noindent
{\it Step 4.} Further, we use the construction of Subsubsection~\ref{cond-fcgp}
to write down geometric conditions for the framed cycles
$C_i(G,\T,P,\L)$ for $i=1,\ldots, N$ that does not contain all vertices of $P$. Recall that the lines of $\L$ here play the role of variables.
}

\vspace{1mm}

{\noindent
{\it Step 5.} Combining together Steps~3 and~4 we get the geometric conditions for the framed cycles
$C_i(G,\T,P,\L)$ (where $i=1,\ldots, N$) in terms of compositions of Operations~I--IV
on the points of $P$ and the lines of $\L(E(G))$.
}

\vspace{1mm}

{\noindent
{\it Step 6.} Finally, we write down the geometric conditions obtained in Step~5 in terms of the corresponding configuration space $\Xi_G$
(following Subsubsection~\ref{wwwww}).
}

\vspace{1mm}

{\noindent
{\bf Output.} As an output we get the system of geometric conditions on the space $\Xi_G$.
By Theorem~\ref{tensegrity=condi} this system is fulfilled if and only if there exists a non-parallelizable tensegrity at $\Xi_G$.
}

\begin{remark}
Since monodromies are multiplicative with respect to cycle addition,
at Step~2 it is sufficient to pick only the simple cycles generating $H_1(G)$.
In practice even less cycles are usually needed for the complete output.
\end{remark}

\section{Conjecture on strong geometric conditions for tensegrities}

In Subsection~7.1 we define strong geometric conditions, show several examples
and formulate a conjecture for tensegrities.
Further, in Subsection~7.2 we say a few words regarding this conjecture for the case of
graphs having less than 11 vertices.

\subsection{Some examples and the conjecture formulation}

First, we define strong geometric conditions.

\begin{definition}
Consider a configuration space $\Xi$ of fixed points and non-fixed lines.
We say that a geometric condition on vertices of $\Xi$ is {\it strong}
if it does not depend on non-fixed lines of $\Xi$
and it does not include any point or line obtained by Operations~III and~IV.
\end{definition}

In order to construct a strong geometric condition one uses only Operations~I and~II.
Therefore, any strong geometric condition is represented by some Cayley algebra expression
(see Remark~\ref{Cayley-remark}).

\begin{remark}
Every strong geometric condition can be represented by some 3-point relation.
It is always possible to do so,
since every line in the condition is defined by a pair of already constructed points (including the points of $P$).
So if $\ell=pq$ then
$$
\ell\cap \ell_2 \cap \ell_3=true \quad \Leftrightarrow \quad (p,q,\ell_2 \cap \ell_3)=true,
$$
and
$$
p_2\in\ell=true \quad \Leftrightarrow \quad (p_2,p,q)=true.
$$
\end{remark}

We continue with several examples, showing different relations between geometric conditions and
strong geometric conditions.

\begin{example}\label{exx1}
Let us start with an example of a strong geometric condition.
Let
$$
\Xi_1=\big((p_1,p_2,\ldots,p_7),()\big)
$$
(here the list $L$ is empty).
Consider the condition:
$$
\big(p_1p_4\cap p_2p_3,p_3p_6\cap p_4p_5,p_7\big)=true.
$$
This condition is a strong geometric condition. Here is a configuration of seven points satisfying this condition:
$$
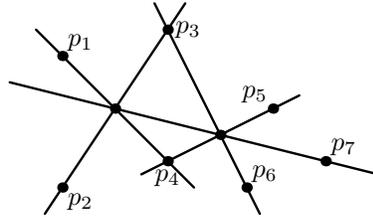
\epsfbox{pic.6}$$
\end{example}

\begin{example}\label{exx2}
In the second example we consider a system of geometric conditions which, in fact, can be rewritten
as a strong geometric condition.
Let
$$
\Xi_2=\big((p_1,p_2,p_3),(\ell_{2,1})\big),
$$
recall that $p_2\in \ell_{2,1}$. Then the system of geometric conditions
$$
(p_2, \ell_{2,1})=true=(p_3,\ell_{2,1})
$$
$$
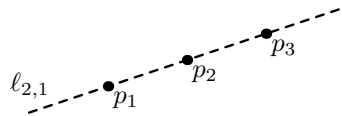
\epsfbox{pic.7}
$$
(here the dashed line is the non-fixed line) is equivalent to the strong geometric condition
$$
(p_1,p_2,p_3)=true.
$$
$$
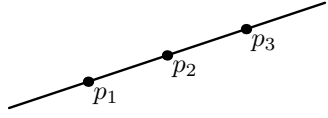
\epsfbox{pic.8}
$$

\end{example}

\begin{example}\label{exx3}
In the last example here we consider a geometric condition which cannot be rewritten in terms of
strong geometric conditions.
Let
$$
\Xi_3=\big((p_1,p_2,p_3,p_4,p_5, p_6), (\ell_{1,1},\ell_{2,1},\ell_{3,1})\big),
$$
recall that $p_1\in \ell_{1,1}$, $p_2\in \ell_{2,1}$, and $p_3\in \ell_{3,1}$.
The system of geometric conditions
$$
\left\{
\begin{array}{l}
\ell_{1,1}\cap\ell_{2,1}\cap p_4p_5=true\\
\ell_{2,1}\cap\ell_{3,1}\cap p_5p_6=true\\
\ell_{3,1}\cap\ell_{1,1}\cap p_6p_4=true\\
\end{array}
\right.
$$
is not equivalent to any system of strong geometric conditions on $P$.
The reason for that is that certain initial generic point configurations $P$ admit no solutions while the others generic ones admit two solutions.
The set configurations admitting solutions is not algebraic (it is semi-algebraic), while all strongly geometric conditions are algebraic
sets. Thus this system cannot be rewritten in terms of strongly geometric conditions (we leave the details as an exercise for the reader).
Below is the example of a configuration of nine points and three lines satisfying the above system of condition:
$$
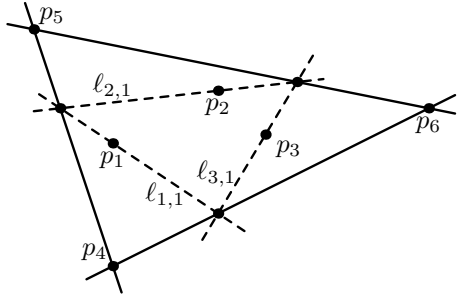
\epsfbox{pic.9}
$$
(here the dashed lines are the non-fixed lines).
\end{example}

Finally, let us formulate the conjecture on strong geometric description of tensegrities
(this is a slightly modified version of Conjecture~2 of~\cite{DKS}).
\begin{conjecture}\label{tensegrity=condi-conjecture}
For every graph $G$ there exists a system of strong geometric conditions such that
a framework $G(P)$ in general position admits a non-parallelizable tensegrity
if and only if
$P$ satisfies this system of strong geometric conditions.
\end{conjecture}

If this conjecture is true then every condition is expressed in terms of Cayley algebra
operations on the vertices of frameworks (see Remark~\ref{Cayley-remark}).
The important additional problem here is to give a constructive description of such conditions.

\subsection{Strong geometric conditions for graphs on small number of vertices}

In this subsection we test Conjecture~\ref{tensegrity=condi-conjecture} on
graphs with $n$ vertices and $2n-3$ edges for $n<10$.

\begin{remark}
It is believed that all geometric conditions for an arbitrary graph
is reduced to the ones coming from graphs with $n$ vertices and $2n-3$ edges.
In fact, the most interesting connected graphs with $n$ vertices and $2n-3$ edges are Laman graphs.
Each subset of $m\ge 2$ vertices of such graph spans at most $2m-3$ edges.
One of the important properties of such graphs is that they can be embedded as pseudo-triangulations (Theorem~1.1 of~\cite{Rot2}).
This is one of the ways to show that such graphs do not have nonzero tensegrities for a randomly chosen configuration of points.
\end{remark}

Let us continue with the following two observations.

\vspace{2mm}

{\bf Observation A.}
If all vertices of a cycle are of degree 3
then the geometric condition defined by this cycle is a strong geometric condition.

\vspace{2mm}

{\bf Observation B.}
Suppose all vertices of a cycle are of degree 3 except for one which is of degree greater than 3.
Then one of the non-fixed lines of the vertex of degree greater than 3 is defined by a composition of
2-point and 2-line operators on vertices of the cycle.
This means that a geometric condition for the graph with $n$ vertices has a reduction to
geometric conditions on a certain graph with $n-1$ vertices. In addition if the graph on $n-1$
vertices is defined by a system of strong geometric conditions,
the original graph is defined by a system of strong geometric conditions as well.

\vspace{2mm}

\begin{lemma}\label{Lemma1}
Consider a connected graph $G$ whose vertices are of degree $\ge 3$,
suppose that $m$ of them are of degree $> 3$ while the rest are of degree $3$.
Let $G$ has $k$ vertices of degree 3 forming a connected component $($denote it by $G_0$$)$.
If $k+2>m$ then $G$ has a simple cycle with at most one vertex of degree $>3$.
\end{lemma}

\begin{proof}
Let $G_0$ be a maximal connected component satisfying the conditions of the lemma.
If $G_0$ has a cycle then by Observation A the geometric condition defined by this cycle is a strong geometric condition.
Suppose now $G_0$ is a tree. Then it has precisely $k+2$ edges connecting $G_0$ with $G\setminus G_0$.
Since $G_0$ is maximal, there is no vertices of degree 3 of $G\setminus G_0$ connected by an edge with $G_0$.
Since the number of vertices of degree $\ge 3$ is less than $k$, there are two edges connecting $G_0$ with the same
vertex of degree $\ge 3$.
Therefore, we have one cycle whose all vertices except one are of degree 3. So we are in position to use Observation~B.
This concludes the proof of the lemma.
\end{proof}

In our paper~\cite{DKS} we have found all the geometric conditions for graphs on $n$ vertices and $2n-3$ edges for  $n=6,7,8$.
It is interesting to admit that Lemma~\ref{Lemma1} and Observations~A and~B are sufficient to write down strong geometric conditions
for all the graphs considered before and also for the new case of graphs on $9$ vertices
(and the corresponding examples of papers~\cite{Guz2,Guz3,Guz1,WW}).
Here is a way to do this.

\vspace{2mm}

{\it A general remark.} Let a graph $G$ has a vertex of degree 1 or 2.
Then $G$ admits a non-parallelizable tensegrity only if
the framework is not in general position. We skip such cases.

\vspace{2mm}

{\it Case $n<6$.} If $n<6$ then $\frac{3}{2}n>2n-3$. Hence we always have some vertex of degree 1 or 2.

\vspace{2mm}

{\it Case $n=6$.} We have 9 edges here. All cycles have all vertices of degree 3. Hence
according to Observation~A we have strong geometric conditions for them.

\vspace{2mm}

{\it Case $n=7$.} We have 11 edges here. All cycles have all vertices of degree 3 except one which is of degree $4$. Hence
according to Observations~A and~B we have strong geometric conditions for them (we either have a condition straight away
or make a reduction to the case of $n=6$).

\vspace{2mm}

{\it Case $n=8$.} We have 13 edges here. Here we have at most two vertices of degree $>3$ and at least 6 vertices of degree 3.
Hence $k\ge 1$ and $m\le 2$. Therefore by Lemma~\ref{Lemma1} we have either a strong geometric condition
or a reduction to the case of $n=7$.

\vspace{2mm}

{\it Case $n=9$.} We have 15 edges here. If we have at most two vertices of degree $>3$ then the situation is as in the case of $n=8$.
The only new case here is if we have three vertices of degree $4$ and six vertices of degree $6$.
In this case there are at most 12 edges adjacent to vertices of degree $4$.
Hence there exists an edge that is adjacent to two vertices of degree 3.
Hence we have a subtree with $k\ge 2$. since $m=3<4\le k+2$ we can apply Lemma~\ref{Lemma1} and reduce this case
to the case of $n=8$.

\vspace{2mm}

We would like to conclude this paper with a small discussion on Conjecture~\ref{tensegrity=condi-conjecture}.
The main evidence that this conjecture is true was provided by numerous different examples confirming it.
However it turns out that all the examples examined satisfy the condition of Lemma~\ref{Lemma1}.
Meanwhile if the number of vertices is great, then the condition of Lemma~\ref{Lemma1} is not necessarily fulfilled.
One might expect a large number of vertices of degree greater than $3$,
this can easily happen for graphs on $n$ vertices and $2n-3$ edges.
The extremal example here is when we have
$n-6$ vertices of degree 4 and six vertices of degree 3.
The case $n=10$ already contains several non-equivalent (up to $\hf$-surgeries) graphs
for which the condition of  Lemma~\ref{Lemma1} is not fulfilled.
Here is an example of one of them:

$$
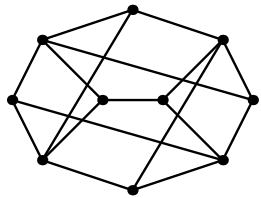
\epsfbox{pic.14}
$$
For this graph we are no longer in position to apply Lemma~\ref{Lemma1},
so this graph is a simplest candidate for a counterexample.

\vspace{5mm}


\begin{thebibliography}{99}
\bibitem{Cas} D.L.D.~Caspar and A. Klug, {\it Physical principles in the
construction of regular viruses}, in Proceedings of Cold Spring
Harbor Symposium on Quantitative Biology, vol.~27 (1962),
pp.~1--24.
%
\bibitem{Che}
J.~Cheng, M.~Sitharam, I.~Streinu
{\it Nucleation-free $3D$ rigidity}, 27~p. (2013), arXiv:1311.4859 [cs.CG].
%
\bibitem{Con} R.~Connelly and W.~Whiteley, {\it Second-order rigidity and
prestress stability for tensegrity frameworks}, SIAM Journal of
Discrete Mathematics, vol.~9, no.~3 (1996), pp.~453--491.
%
\bibitem{Con2}
R.~Connelly, {\it Tensegrities and global rigidity},
Shaping space, Springer, New York (2013),  pp.~267-–278.
%
\bibitem{Con2013}
R.~Connelly, {\it What is $\ldots$ a tensegrity?}, Notices Amer. Math. Soc., vol.~60, no.~1 (2013),  pp.~78-–80.
%
\bibitem{DKS}
F.~Doray, O.~Karpenkov, J.~Schepers,
{\it Geometry of configuration spaces of tensegrities},
Disc. Comp. Geom., vol.~43, no.~2 (2010), pp.~436--466.
%
\bibitem{DRS}
P.~Doubilet, G.-C.~Rota, J.~Stein,
{\it On the foundations of combinatorial theory.
IX. Combinatorial methods in invariant theory},
Studies in Appl. Math., vol.~53 (1974), pp.~185-–216.
%
\bibitem{Rot2}
R.~Haas, D.~Orden, G.~Rote, F.~Santos, B.~Servatius, H.~Servatius,
D.~Souvaine, I.~Streinu, W.~Whiteley, {\it Planar minimally rigid
graphs and pseudo-triangulations}, Comput. Geom., vol.~31, no.~1--2
(2005), pp.~31--61.
%
\bibitem{Guz2} M.~de~Guzm\'an, {\it Finding Tensegrity Forms}, preprint (2004).
%
\bibitem{Guz3} M.~de~Guzm\'an, D. Orden, {\it Finding tensegrity
structures: Geometric and symbolic aproaches}, in Proceedings of
EACA-2004, pp.~167--172 (2004).
%
\bibitem{Guz1} M.~de~Guzm\'an, D.~Orden, {\it From graphs to tensegrity structures:
Geometric and symbolic approaches}, Publ. Mat. 50  (2006),
pp.~279--299.
%
\bibitem{Gra}
D.~Gray, http://expedition.uk.com/projects/tensegritree-university-of-kent/
%
\bibitem{Ing} D.~E.~Ingber, {\it Cellular tensegrity: defining new rules of
biological design that govern the cytoskeleton}, Journal of Cell
Science, vol.~104 (1993), pp.~613--627.
%
\bibitem{Ing2}
D.E.~Ingber, N.~Wang, D.~Stamenovi\'c,
{\it Tensegrity, cellular biophysics, and the mechanics of living systems},
Rep. Progr. Phys., vol.~77, no.~4 (2014), 046603, 21 p.
%
\bibitem{JJSS2015}
B.~Jackson, T.~Jord\'an, B.~Servatius, H.~Servatius,
{\it Henneberg moves on mechanisms},
Beitr. Algebra Geom.,
vol.~56, no.~2 (2015),  pp.~587-–591.
%
\bibitem{JN2015}
B.~Jackson, A.~Nixon,
{\it Stress matrices and global rigidity of frameworks on surfaces},
Discrete Comput. Geom. vol.~54, no.~3 (2015), pp.~586-–609.
%
\bibitem{KSS}
O.~Karpenkov, J.~Schepers, B.~Servatius,
{\it On stratifications for planar tensegrities with a small number of vertices},
ARS Mathematica Contemporanea, vol.~6, no.~2 (2013), pp.~305--322.
%
\bibitem{Kit2014}
D.~Kitson, S.C.~Power,
{\it Infinitesimal rigidity for non-Euclidean bar-joint frameworks},
Bull. Lond. Math. Soc., vol.~46, no.~4, (2014),  pp.~685-–697.
%
\bibitem{Kit2015}
D.~Kitson, B.~Schulze
{\it Maxwell-Laman counts for bar-joint frameworks in normed spaces},
Linear Algebra and its Applications vol.~481, (2015), pp.~313--329.
%
\bibitem{HLi}
H.~Li, {\it Invariant algebras and geometric reasoning.}
With a foreword by David Hestenes. World Scientific Publishing Co. Pte. Ltd., Hackensack, NJ, 2008, xiv+518 pp.
%
\bibitem{Max} J.~C.~Maxwell, {\it On reciprocal figures and diagrams of
forces}, Philos. Mag. vol.~4, no.~27 (1864), pp.~250--261.
%
\bibitem{Mot} R.~Motro, {\it Tensegrity: Structural systems for the future}, Kogan
Page Science, London, 2003.
%
%
\bibitem{Rot1} B.~Roth, W.~Whiteley, {\it Tensegrity frameworks},
Trans.\ Amer.\ Math.\ Soc., vol.~265, no.~2 (1981), pp.~419--446.
%
\bibitem{SW} F.V.~Saliola, W.~Whiteley,
{\it Some notes on the equivalence of first-order rigidity in various geometries},
arXiv:0709.3354 [math.MG], 15~p.
%
\bibitem{Sim}
C.~Simona-Mariana, B.~Gabriela-Catalina, {\it Tensegrity applied to modelling the motion of viruses},
Acta Mech. Sin. vol.~27, no.~1 (2011), pp.~125-–129.
%
\bibitem{Ske} R.E.~Skelton, {\it Deployable tendon-controlled structure}, United
States Patent 5642590, July 1, 1997.
%
\bibitem{Sne} K.~Snelson, http://www.kennethsnelson.net.
%
%
\bibitem{Sta}
R.P.~Stanley, {\it Enumerative combinatorics. Vol. 2.} With a foreword by Gian-Carlo Rota and appendix 1 by Sergey Fomin.
Cambridge Studies in Advanced Mathematics, 62. Cambridge University Press, Cambridge, 1999, xii+581 pp.
%
\bibitem{Tib} A.~G.~Tibert. {\it Deployable tensegrity structures for space
applications}, Ph.D. Thesis, Royal Institute of Technology,
Stokholm 2002.
%
\bibitem{Wan}
M.~Wang, M.~Sitharam,
{\it Combinatorial Rigidity and Independence of Generalized Pinned Subspace-Incidence
               Constraint Systems}, in Automated Deduction in Geometry - 10th International
Workshop, Coimbra, Portugal (2014), pp.~166--180, arXiv:1503.01837 [cs.CG].
%
\bibitem{WW} N.~L.~White, W.~Whiteley,
{\it The algebraic geometry of stresses in frameworks}, SIAM J.
Alg. Disc. Meth., vol.~4, no.~4 (1983), pp.~481--511.
%
\bibitem{WW2} N.~L.~White, W.~Whiteley,
{\it The algebraic geometry of motions of bar-and-body frameworks},
SIAM J. Algebraic Discrete Methods, vol.~8, no. 1 (1987), pp.~1–-32.
%
\bibitem{WM}
N.L.~White, T.~McMillan, {\it Cayley factorization. Symbolic and algebraic computation}, Rome, (1988), pp.~521-–533,
Lecture Notes in Comput. Sci., 358, Springer, Berlin, 1989.
%
\bibitem{White}
N. White,
{\it The bracket ring of a combinatorial geometry},
I, Trans. Amer. Math. Soc., vol.~202 (1975),
pp. 79--95.
%
\bibitem{Whi} W.~Whiteley, {\it Rigidity and scene analysis}, in J.E.~Goodman and
J.~O'Rourke, editors, Handbook of Discrete and Computational
Geometry, chapt.~49, pp.~893--916, CRC Press, New York, 1997.
\end{thebibliography}
\end{document}